\documentclass[final,leqno,onefignum,onetabnum]{siamltex1213}

\usepackage{color}
\usepackage{amsmath}
\usepackage{amsfonts}
\usepackage{amssymb}
\usepackage{multirow}
\usepackage[dvips]{graphicx}
\usepackage{epsfig}
\usepackage{latexsym}
\usepackage{enumerate}
\usepackage{tabularx}
\usepackage{booktabs}
\usepackage[caption=false]{subfig}

\newtheorem{thm}{Theorem}[section]
\newtheorem{thm2}{Theorem}
\newtheorem{thm3}{Theorem}
\newtheorem{thm4}{Theorem}
\newtheorem{thm5}{Theorem}

\newtheorem{lem}[thm]{Lemma}
\newtheorem{prop}[thm]{Proposition}
\newtheorem{defn}[thm5]{Definition}
\newtheorem{example}[thm4]{Example}
\newtheorem{rem}[thm3]{Remark}
\newtheorem{algorithm}[thm2]{Algorithm}

\def \diag {\,{\rm diag}\,}

\def \argmin {\,{\rm argmin}\,}

\newcommand{\se}{\text{e}}

\newcommand{\bR}{\mathbb{R}}

\newcommand{\bN}{\mathbb{N}}
\newcommand{\bZ}{\mathbb{Z}}

\newcommand{\vc}{\boldsymbol{c}}
\newcommand{\vd}{\boldsymbol{d}}

\newcommand{\valpha}{\boldsymbol{\alpha}}

\newcommand{\vhf}{\boldsymbol{\hat{f}}}
\newcommand{\vhh}{\boldsymbol{\hat{h}}}

\newcommand{\cH}{\mathcal{H}}

\newcommand{\mB}{\mathsf{B}}

\newcommand{\mT}{\mathsf{T}}
\newcommand{\mLambda}{\mathsf{D}}
\newcommand{\md}{\mathsf{d}}
\newcommand{\mPsi}{\mathsf{\Psi}}

\newcommand{\mOmega}{\mathsf{\Omega}}

\newcommand{\tA}{\tilde{A}}

\newcommand{\hf}{\hat{f}}
\newcommand{\hvf}{\boldsymbol{\hat{f}}}

\title{A Frame Theoretic Approach to the Non-Uniform Fast Fourier Transform \thanks{This work is supported in part by grants NSF-DMS 1216559 and AFOSR 12004863.}}

\author{Anne Gelb \thanks{School of Mathematical and Statistical Sciences,
         Arizona State University,
         P.O. Box 871804,
         Tempe, AZ 85287-1804. (\email{anne.gelb@asu.edu})} 
\and Guohui Song \thanks{Department of Mathematics, Clarkson University, Potsdam, NY 13699. (\email{gsong@clarkson.edu})}
}

\begin{document}
\maketitle

\begin{abstract}
Nonuniform Fourier data are routinely collected in applications such as magnetic resonance imaging, synthetic aperture radar, and synthetic imaging in radio astronomy.  To acquire a fast reconstruction that does not require an online inverse process, the non-uniform fast Fourier transform (NFFT), also called convolutional gridding,  is frequently employed.  While various investigations have led to improvements in accuracy, efficiency, and robustness of the NFFT, not much attention has been paid to the fundamental analysis of the scheme, and in particular its convergence properties.  This paper analyzes the convergence of the NFFT by casting it as a Fourier frame approximation.  In so doing, we are able to design parameters for the method that satisfy conditions for numerical convergence.  Our so called frame theoretic convolutional gridding algorithm can also be applied to detect features (such as edges) from non-uniform Fourier samples of piecewise smooth functions.
\end{abstract}



\pagestyle{myheadings}
\thispagestyle{plain}
\markboth{A. Gelb AND G. Song}{A Frame Theoretic Approach to the NFFT} 

\section{Introduction}
\label{sec:intro}
In several imaging applications, such as magnetic resonance imaging (MRI), \cite{BW00, pipe, sedarat}, synthetic aperture radar (SAR), or synthetic imaging in radio astronomy, \cite{briggs}, data are finitely sampled in the Fourier domain.  Although in many cases the data are presumed to be sampled uniformly, there are some situations for which the data are collected non-uniformly, either by instrument design or by the circumstances surrounding the collection.  Regardless of the particular collection protocol, reconstructing the underlying image can be viewed as an approximation to the inverse Fourier transform.   As the data often must be accurately analyzed in real time, numerical efficiency and robustness of the numerical algorithm is essential.

One common technique used to reconstruct images from non-uniform finitely sampled Fourier data involves ``regridding'' the data uniformly, via a convolution, so that the inverse fast Fourier transform (FFT) algorithm can be applied.  This technique goes by various names, depending on the particular application.  We will use the terms {\em non-uniform fast Fourier transform} (NFFT), \cite{nufft, fourmont}, and {\em convolutional gridding}, \cite{jackson, o'sullivan, pipe, sedarat}, interchangeably.   Briefly, the method uses practical and heuristic arguments to construct an ``essentially'' compactly supported window function, which is also ``essentially'' compactly supported in the Fourier domain.  Uniform Fourier coefficients are then computed by numerically approximating the convolution of the Fourier transform of the underlying function with the window function from the given non-uniformly sampled Fourier data.   Numerical weights, often called {\em density compensation factors}, are introduced to evaluate the convolution integral via summation.  The FFT is then used directly to reconstruct the underlying image.  

Although the convolutional gridding algorithm typically achieves satisfactory results, it is not difficult to generate examples where the method fails, \cite{Aditya09}.  Specifically, if the density compensation factors are viewed as quadrature weights for the inverse Fourier transform (or convolution integral), then the rate of convergence is tied to the accuracy of the numerical quadrature scheme.  Moreover, for sampling patterns that are increasingly sparse in the high frequency domain, such as those found in MRI, increasing the number of data points does nothing to improve the resolution, and hence standard quadrature rules may fail to converge.
Thus developing a new rigorous mathematical framework to analyze the convergence properties of the convolutional gridding technique is highly desirable. 

There are, of course, other algorithms that recover $f$ from its non-uniform Fourier data.  For example, a well studied problem in sampling theory is to recover a band-limited $\hat{f}$ from non-uniform samples  (see e.g. ~\cite{aldroubi2001nonuniform, benedetto1990irregular, feichtinger1995efficient,  feichtinger1992irregular} for history and references).  Techniques using compressed sensing are also becoming more prevalent, e.g. ~\cite{LDP}.  While these approaches each have their advantages, they are primarily solved using iterative algorithms, thus making them somewhat difficult to analyze.  The purpose of this paper is to establish a theoretical framework for the forward convolutional gridding method.
Specifically, instead of viewing convolutional gridding essentially as an interpolation-numerical quadrature technique, i.e., as a means to approximate uniform Fourier coefficients and then perform the (inverse) FFT, we will consider the convolutional gridding algorithm to be a {\em Fourier frame approximation} of the underlying function.\footnote{We note that in \cite{BW00} it was suggested that certain MRI sampling patterns may indeed be viewed as Fourier frame sequences.} To do this, we incorporate the convolutional gridding window function into the representation of $f$.  That is, we express $f$ in terms of the Fourier basis divided by the window function (rather than the standard Fourier basis).  There are several advantages to our approach.  First, it completely eliminates the intermediate interpolation step and the errors associated with it.  Second, it provides a framework for constructing both the density compensation factors and the window function using rigorous mathematical arguments while satisfying the need for efficient computation.  Finally, it establishes a forward algorithm for the finite (Fourier) frame approximation, which may be useful in other applications where data are sampled in a non-standard way. 


Our paper is organized as follows.  In Section \ref{sec:convgriddingandframe} we review the convolutional gridding algorithm and link it to the Fourier frame approximation.  We demonstrate how the density compensation factors for the convolutional gridding algorithm can be chosen to generate the best finite frame approximation from the given Fourier frame data.  In Section \ref{sec:edgedetection} we illustrate how our frame theoretic convolutional gridding algorithm can also be used to detect edges of a piecewise smooth function from non-uniform Fourier data.  Numerical examples are given in Section \ref{sec:numerics}, and in Section \ref{sec:conclusion} we provide some concluding remarks.

\section{Convolutional Gridding and the Finite Frame Approximation}
\label{sec:convgriddingandframe}
\subsection{The Convolutional Gridding Algorithm}
\label{sec:convgridding}

Let $f$ be an unknown piecewise analytic function supported in $[0, 1]$. Suppose we are given the finite samples 
\begin{eqnarray}
\label{eq:foursamples}
\hf_j=\langle f, \varphi_j \rangle = \int_0^1 f(x) \se^{-2\pi i \lambda_j x}dx, \quad -n\leq j\leq n,
\end{eqnarray}
where $\{\lambda_j: j\in \bZ\}$ is a (non-uniform) sequence in $\bR$. Our goal is to recover the underlying function $f$.

For illustrative purposes we  will consider the following sampling schemes for $\{\hat{f}(\lambda_j)\}_{j = -n}^n$, depicted in Figure \ref{fig:sampling}:
\begin{figure}[htbp]
        \subfloat[Jittered Sampling]{\includegraphics[scale = .40]{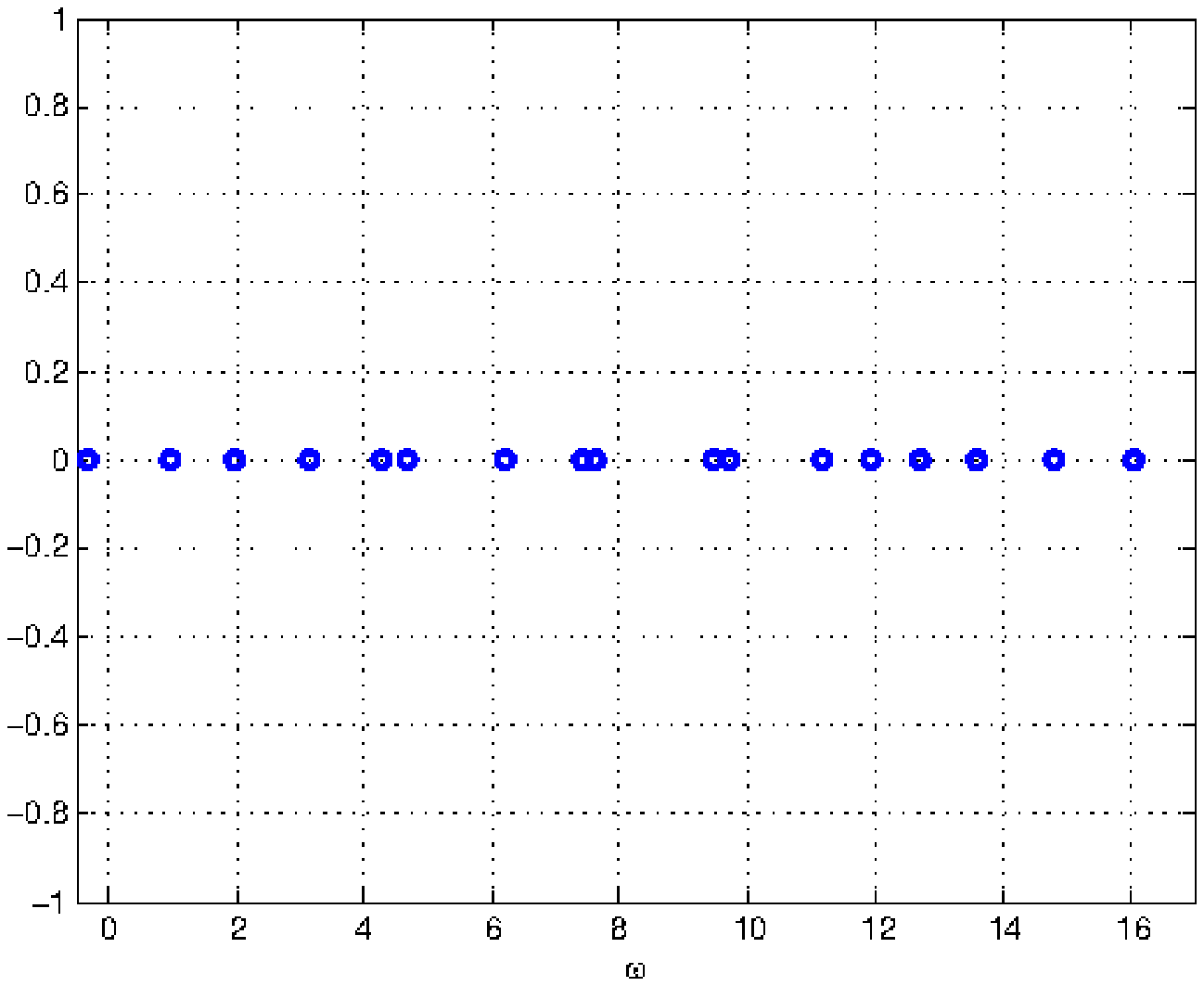}}
        \qquad
        \subfloat[Log Sampling]{ \includegraphics[scale = .40]{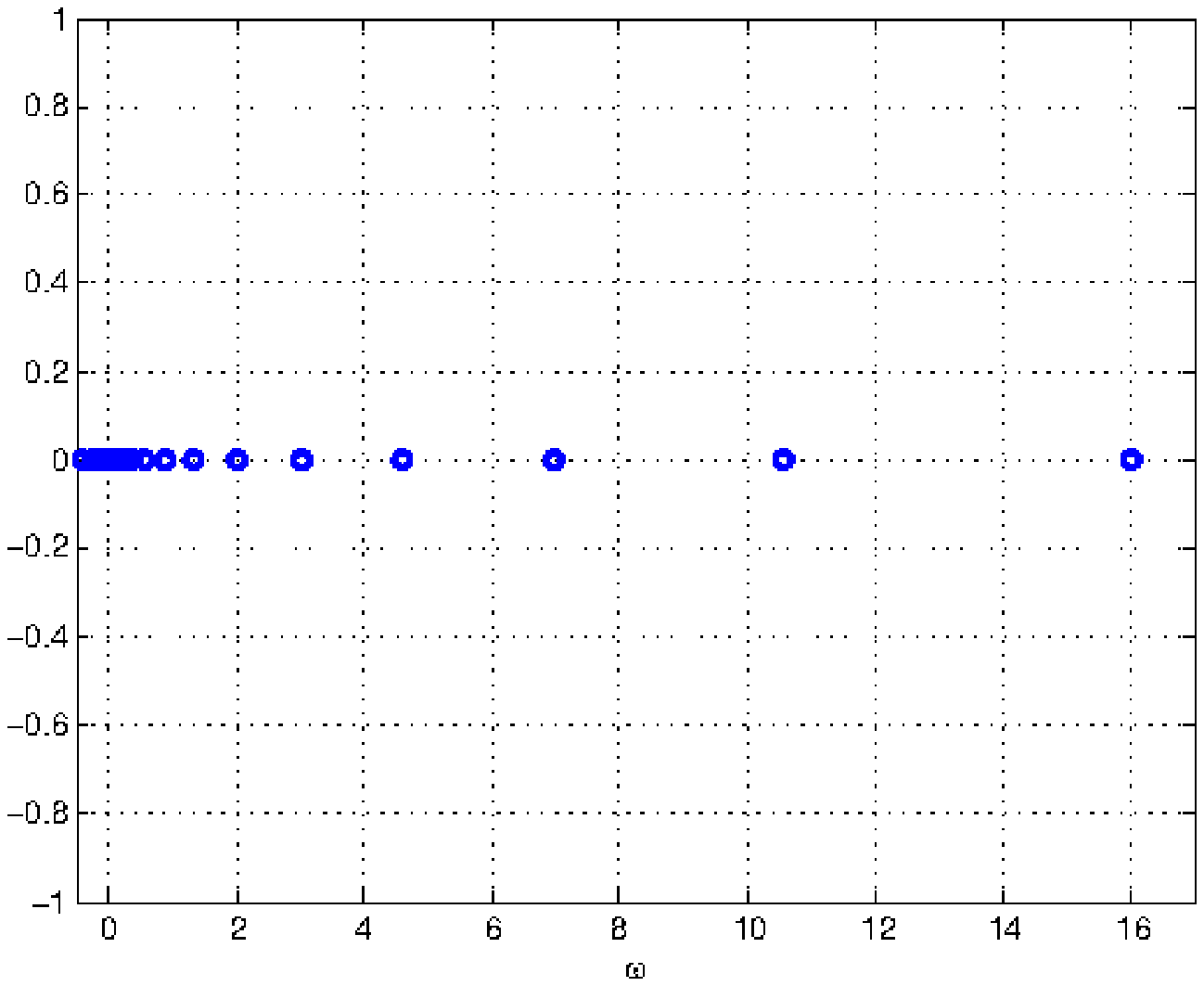}}
        \caption{Non-uniform Sampling Schemes (right half plane), $N=16$}
\label{fig:sampling}
\end{figure}
\begin{enumerate}
\item Jittered Sampling:
\begin{equation}
        \lambda_{k} = k \pm \tau_k, \quad \tau_k \sim U[0,\theta], \, k = -n, -(n-1), ... , n.
        \label{eq:jit_samp_exp}
\end{equation}
Here $U[0,\theta]$ denotes a uniform distribution on the interval $[0,\theta]$. The $\tau_k$'s are independent, identically distributed (iid) random variables, and represent a uniform jitter about the equispaced nodes with a maximal jitter of $\theta$. Further, both positive and negative jitters are equiprobable, with the sign of jitter at each node being independent of the sign of jitter at any other node.  We will always assume that $|\theta| < 1/2$.  One such example occurs is MRI, where Cartesian (equispaced) scans may have stochastic jitters in the scanning trajectory caused by the inaccuracies in generating the magnetic field gradients.

\item Log Sampling:  The samples are acquired at logarithmic intervals, with more samples acquired in lower frequencies. Specifically, $|\lambda_{k}|$ is logarithmically distributed between $10^{-v}$ and $n$, with $v>0$ and $2n+1$ being the total number of samples. This sampling pattern is inspired by non-Cartesian sampling patterns in MRI, where typical data acquisition schemes oversample the low frequencies while undersampling the high frequencies.

\end{enumerate}

To motivate the convolutional gridding method, we begin by writing $f$ as its inverse Fourier transform,
\begin{equation}
\label{eq:invFour}
f(x) = \int_{-\infty}^\infty\hat{f}(\lambda) e^{2\pi i \lambda x} d\lambda.
\end{equation}
Since $f$ is assumed to be piecewise analytic in $[0,1]$, $\hat{f}$ can not be  compactly supported.  However, our finite samples lie in the domain $[\lambda_{-n}, \lambda_{n}]$, so the best we can hope to approximate is
\begin{equation}
\label{eq:invFour2}
\tilde{f}(x) = \int_{\lambda_{-n}}^{\lambda_n}\hat{f}(\lambda) e^{2\pi i \lambda x} d\lambda,
\end{equation}
Unfortunately, finding a good quadrature rule for (\ref{eq:invFour2}) is difficult,  due to the oscillatory nature and slow decay of $\hat{f}(\lambda)$ as $\lambda \rightarrow \infty$.  Moreover, the set of samples $\{\hat{f}(\lambda_j)\}_{j = -n}^n$ are pre-determined, so we are not at liberty to choose them.

The approximation of (\ref{eq:invFour2}) may take the form
\begin{equation}
        S_n\tilde{f}(x):= \sum_{k=-n}^n\alpha_k \hat{f}(\lambda_k) e^{2\pi i \lambda_k x},
        \label{eq:dsum}
\end{equation}
where the weights, $\alpha_k$, are known as the {\em density compensation factors (DCFs)} in the MR imaging community.  For example, the DCFs may be generated using the well known trapezoidal rule:
\begin{equation}
        \alpha_k = \lambda_{k+1} - \lambda_k.
        \label{eq:simple_dcf}
\end{equation}
In two dimensions, we may compute a Voronoi tessellation of the non-Cartesian nodes and then assign individual cell areas as the density compensation weights.  For non-Cartesian sampling patterns which intentionally oversample the low frequencies while undersampling the high frequencies, the approximation (\ref{eq:dsum}), where the DCFs are chosen by {\em any} quadrature rule, e.g.~ (\ref{eq:simple_dcf}), will fail to converge.  That is because as the sampling is increased the space between the samples does not decrease, i.e., the sampling is not refined, and standard quadrature rules are therefore not applicable.  Moreover, in approximating (\ref{eq:invFour2}) instead of (\ref{eq:invFour}), we may incur significant error. As noted above, the underlying function is assumed to be piecewise smooth and supported in $[0,1]$, and hence it has associated slow decay rate in the Fourier domain.  If in addition the function exhibits highly oscillatory behavior, even more energy will be lost in the discarded high frequency modes.

This issue is partially addressed in \cite{GelbHines, Aditya09}, where spectral reprojection, \cite{GSSV92},  was applied to modified versions of (\ref{eq:dsum}).   In essence, it is possible to reproject (\ref{eq:dsum}) onto a polynomial basis for which the high frequency information beyond $|\lambda_n|$ contributes only exponentially small values onto the low modes in the reprojection basis.  It is important to note, however, that standard filtering algorithms are not  effective in reducing the error.  This is because classical filters are designed to mollify the high frequency coefficients in the harmonic series, that is, for $\lambda_j = j$, and indeed the numerical convergence estimates rely on the orthogonality of the harmonic basis, \cite{tadmor2007fma}.

Regardless of how the DCFs are chosen, the computational cost associated with evaluating (\ref{eq:dsum}) directly is significant, $\mathcal{O}(N^2)$, and even more so in multi-dimensions, since the fast Fourier transform (FFT) can not be directly applied.   Hence a speedup mechanism was devised to compute non-uniform sums such as (\ref{eq:dsum}) efficiently. The procedure goes by several names -- the non-uniform FFT, \cite{nufft}, the non-equispaced FFT, \cite{fourmont}, FFTs for non-uniform grids, \cite{steidl}, or convolutional gridding in the MR imaging community, \cite{jackson, o'sullivan}. We note also that iterative algorithms have become increasingly popular in computing the NFFT, \cite{beylkin1995fast, dutt1993fast, keiner2009using, potts2001fast, potts2008numerical}. However, the corresponding computational costs may still be too expensive for some applications, such as MRI.  Moreover, it is difficult to analyze the convergence properties of most iterative algorithms.  Thus we are motivated to both analyze and improve the direct methods of computing NFFT, which is the main purpose of this investigation.

The convolutional gridding procedure involves moving the non-uniform measurements, $\hat{f}(\lambda_j)$, to an equispaced grid (gridding) via a convolution operation. To accomplish this, we define $w$ to be a smooth window function and consider the function $g = fw$.  In addition, $w > 0$ should be essentially compactly supported so as to avoid significant aliasing error in the reconstruction of $f = g/w$. 
Some popular choices for window functions include Gaussian and Kaiser Bessel functions, \cite{jackson, o'sullivan, pipe}. 

The standard Fourier expansion is used to obtain $g$: 
\begin{equation}
g(x)=\sum_{l\in \bZ} \hat{g}(l) \se^{2\pi ilx},
\label{eq:Fouriersumg}
\end{equation}
where $\hat{g}(l)=\int_{0}^1 g(x)\se^{-2\pi i l x}dx$ and is determined by the convolution
\begin{equation}
\hat{g}(l) = (\hat{f}*\hat{w})(l) = \int_{-\infty}^{\infty} \hat{f}(\lambda) \hat{w}(l-\lambda) d \lambda.
\label{eq:gconvolve}
\end{equation}
Recall that one of the difficulties in approximating (\ref{eq:invFour}) was the slow decay rate of $\hat{f}$.  Convolutional gridding attempts to address this issue by requiring $\hat{w}$ to be essentially compactly supported.  True compact support is of course impossible if the window function $w$ is also strictly compactly supported.  One motivation for choosing $w$ as either the Kaiser Bessel or Gaussian function is that their Fourier coefficients are explicitly known. Any similar mollifier (bump function) should suffice, however, and some may have better convergence properties.

We can now approximate $\hat{g}(l)$ from the given samples $\{\hat{f}(\lambda_j)\}_{j = -n}^n$ using a quadrature rule of the form
\begin{equation}
\hat{g}(l) \approx \sum_{|j|\leq n}\alpha_j \hat{f}(\lambda_j) \hat{w}(l-\lambda_j).
\label{eq:ghatl}
\end{equation}
We note that generating the optimal DCFs,  $\{\alpha_j\}_{j = -n}^n$, is an ongoing research problem, since  not surprisingly, using quadrature weights such as (\ref{eq:simple_dcf}) to solve (\ref{eq:ghatl}) may not be satisfactory for many applications.  Iterative algorithms that take into consideration the full construction of (\ref{eq:g_approx}) seem to yield the best accuracy, see e.g. \cite{jacobi, pipe, sedarat}. As a consequence of our frame theoretic approach to the convolutional gridding algorithm, a  new technique to determine DCFs will be presented later in Subsection \ref{sec:linkingmethods}.

Substituting (\ref{eq:ghatl}) into (\ref{eq:Fouriersumg}) yields
\begin{equation}\label{eq:g_approx}
g(x) \approx \sum_{l\in \bZ}  \sum_{|j|\leq n}\alpha_j \hat{f}(\lambda_j) \hat{w}(l-\lambda_j) \se^{2\pi ilx},
\end{equation}
and division by the window function $w$ provides an approximation to $f$ as
\begin{equation}
f(x)\approx \sum_{l\in \bZ}  \sum_{|j|\leq n}\alpha_j \hat{f}(\lambda_j) \hat{w}(l-\lambda_j) \frac{\se^{2\pi ilx}}{w(x)}.
\label{eq:fapprox}
\end{equation}
Finally, since $\hat{w}$ is deliberately chosen to have fast decay, (\ref{eq:fapprox}) can be truncated for $l$, leading to the practical {\em convolutional gridding} computation for $f$,
\begin{equation}\label{f-cg}
A_{cg}(f)(x)=   \sum_{|j|\leq n}\,\, \sum_{|l-\lambda_j|\leq q}\alpha_j \hat{f}(\lambda_j) \hat{w}(l-\lambda_j) \frac{\se^{2\pi ilx}}{w(x)},
\end{equation}
where $q$ is a truncation threshold to be specified.  In practical applications, the reconstruction is often ``zero padded'' so that the computation of (\ref{f-cg}) near the boundary is not performed.  This prevents dividing by very small values of $w$.

The steps of the convolutional gridding algorithm are enumerated in Algorithm \ref{alg:grid} while a graphical illustration of the FFT reconstruction and window compensation is provided in Figure \ref{fig:cg}.

\begin{figure}[htbp]
        \subfloat[$\hat{f}$]{\includegraphics[scale = .30]{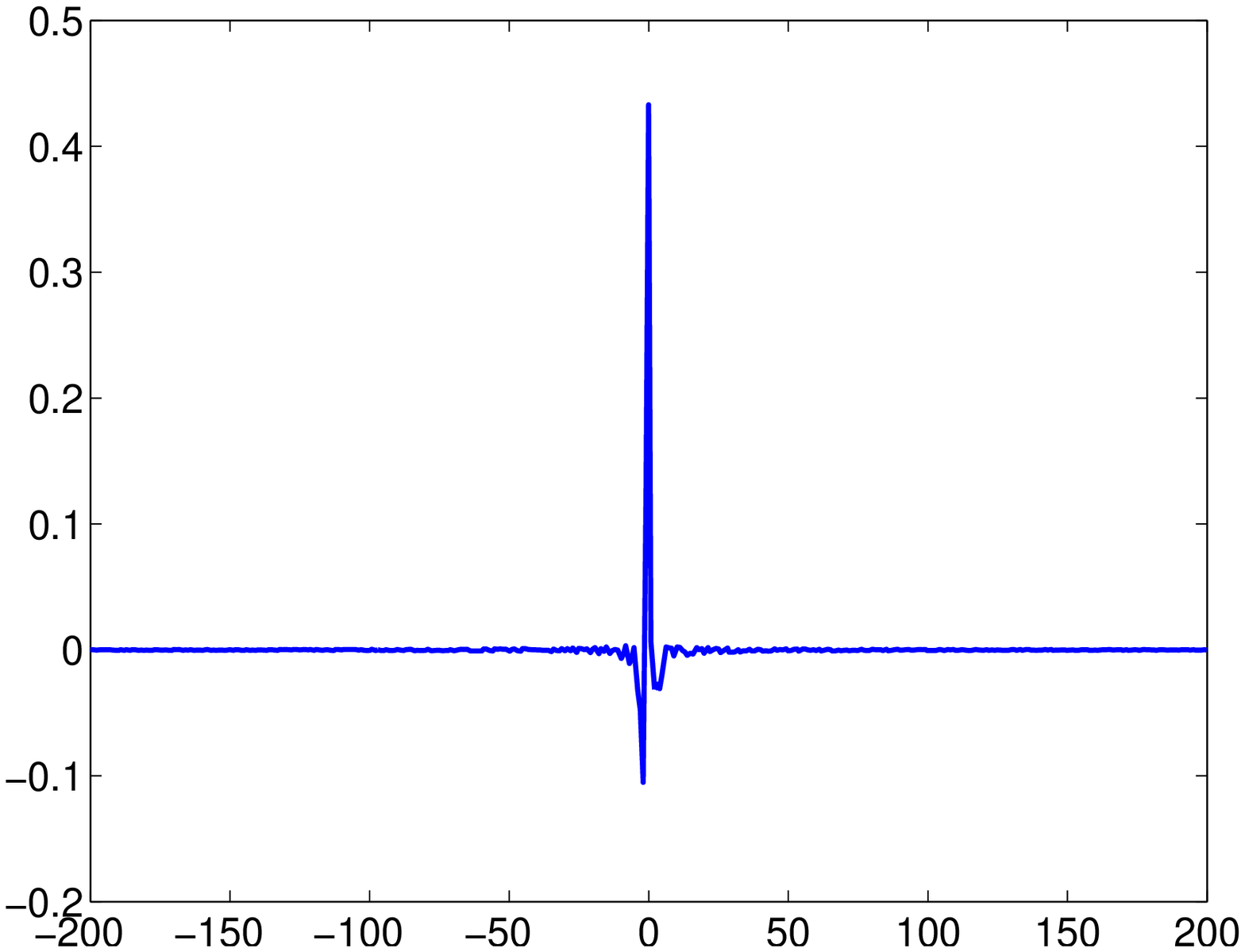}}
        \qquad
        \subfloat[$\hat{w}$]{\includegraphics[scale = .30]{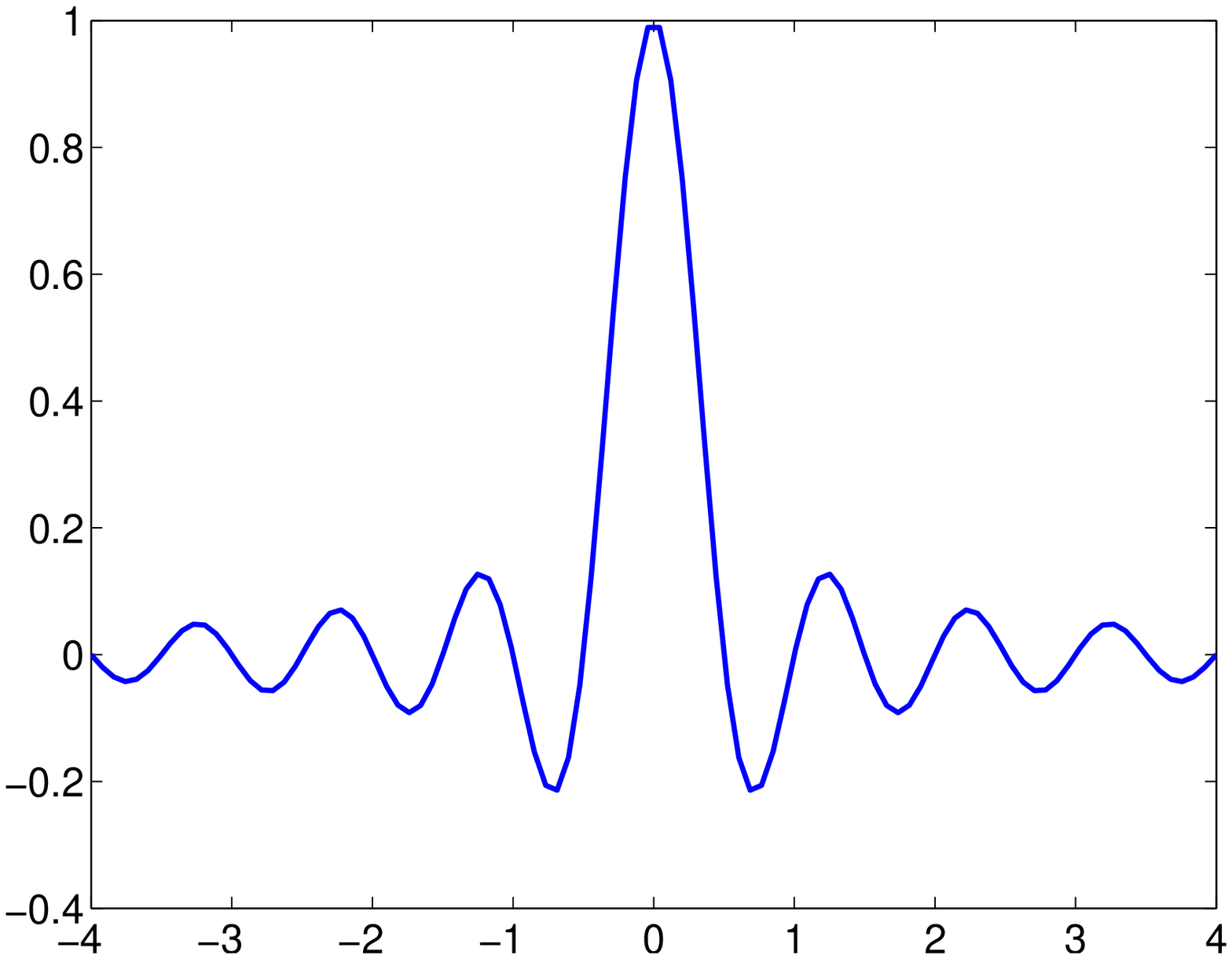}}\\
        \subfloat[$S_N g$]{\includegraphics[scale = .30]{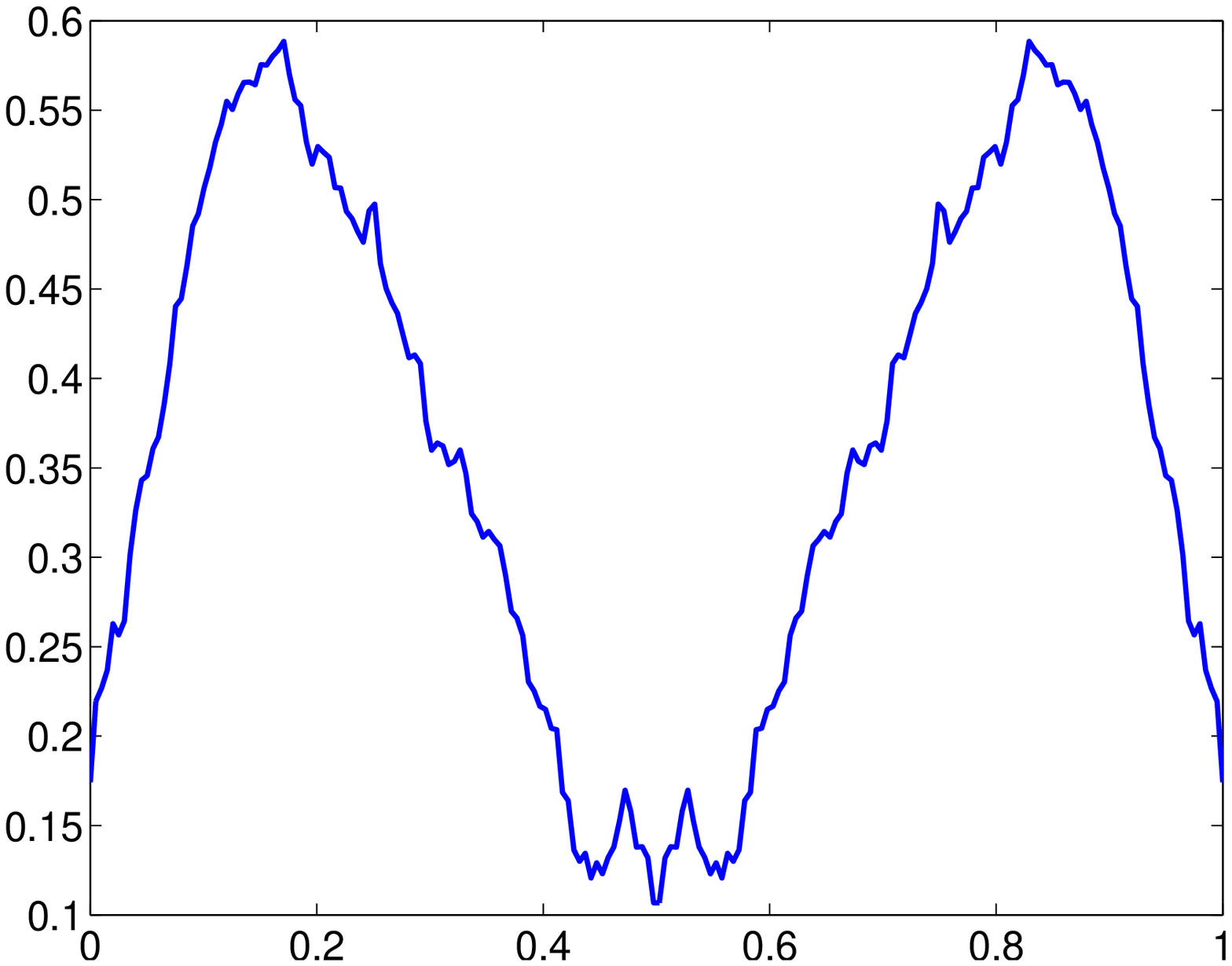}}
        \qquad
        \subfloat[$w$]{\includegraphics[scale = .30]{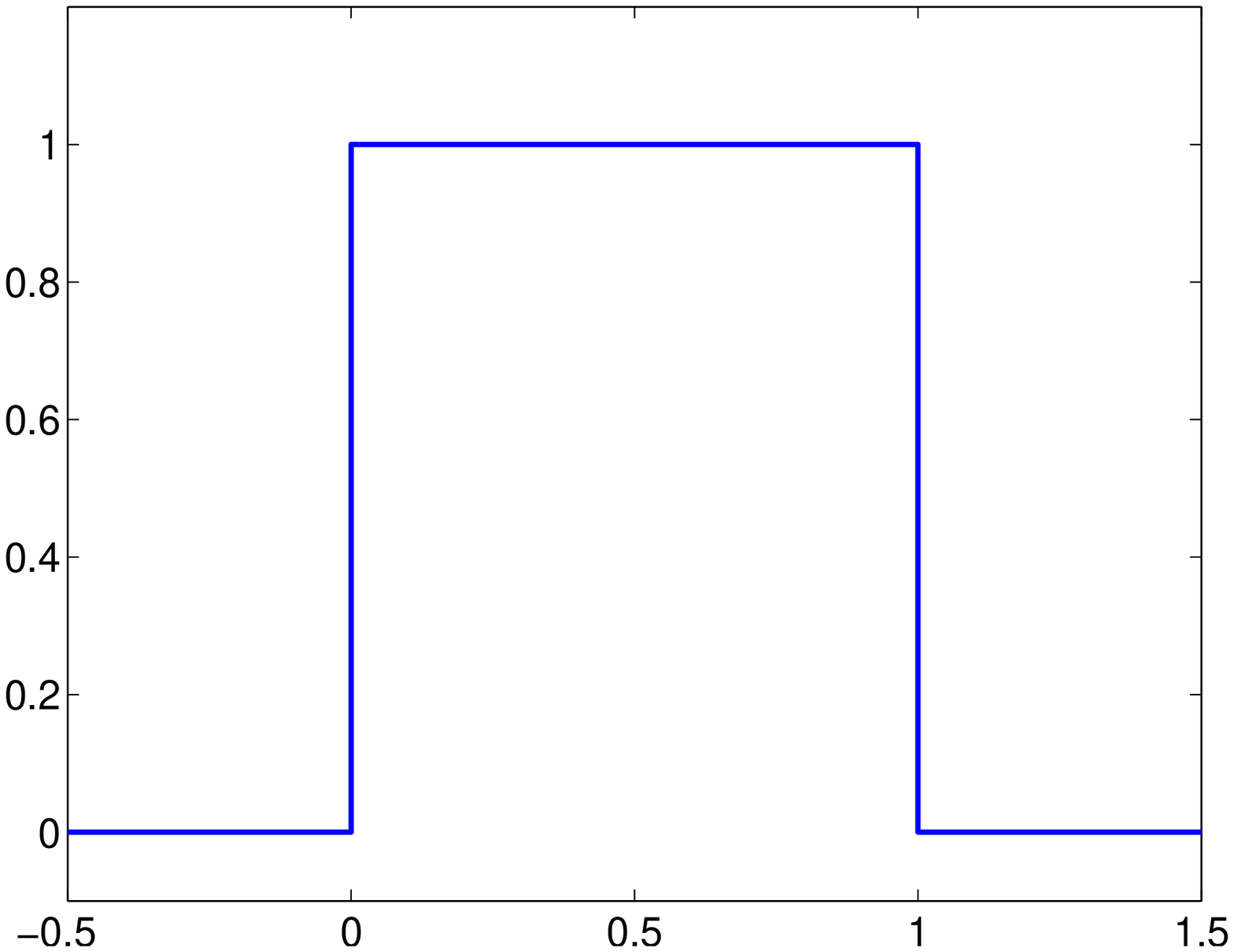}}\\
        \subfloat[$\frac{S_N g}{w}$]{\includegraphics[scale = .30]{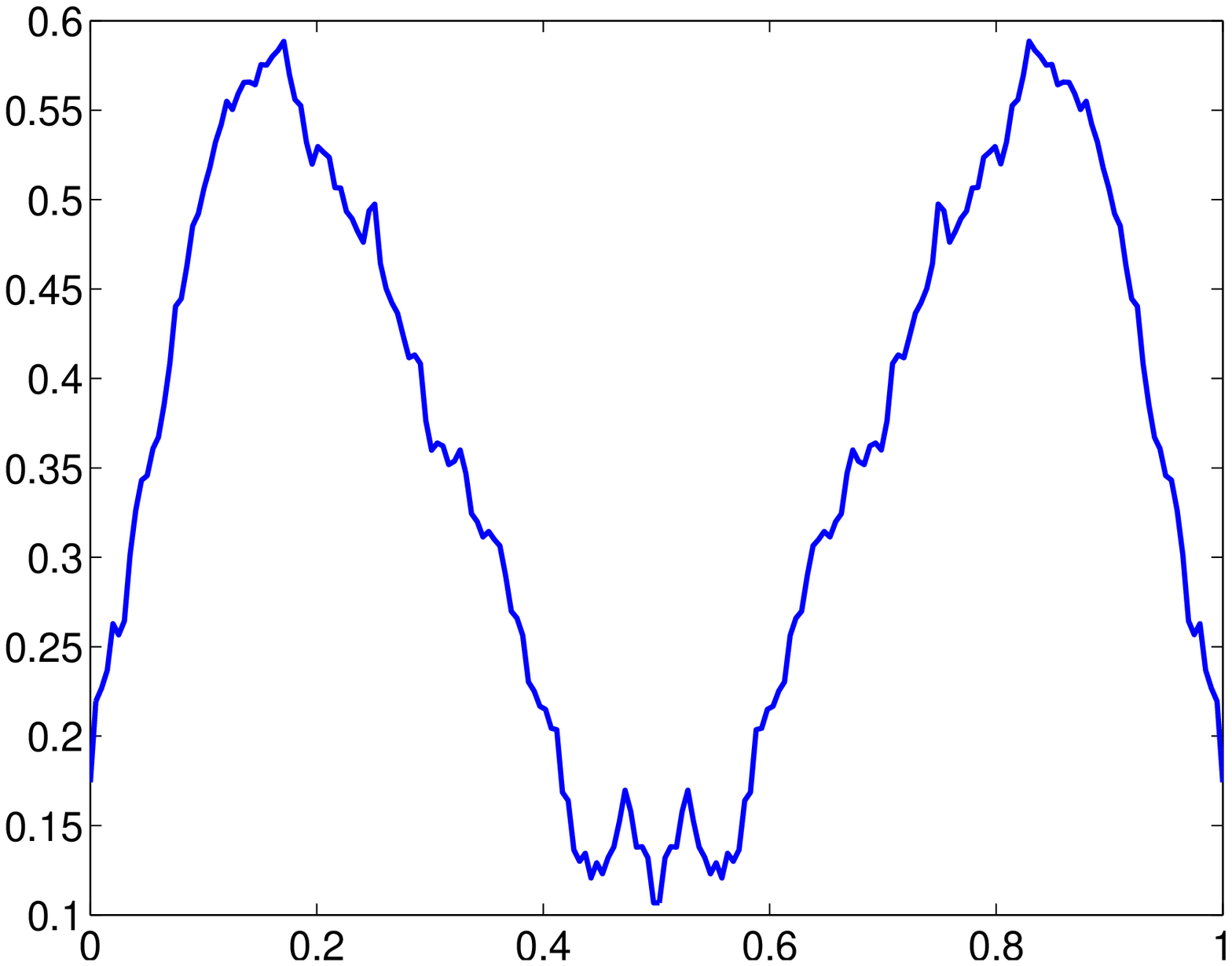}}
        \qquad
        \subfloat[$f$]{\includegraphics[scale = .30]{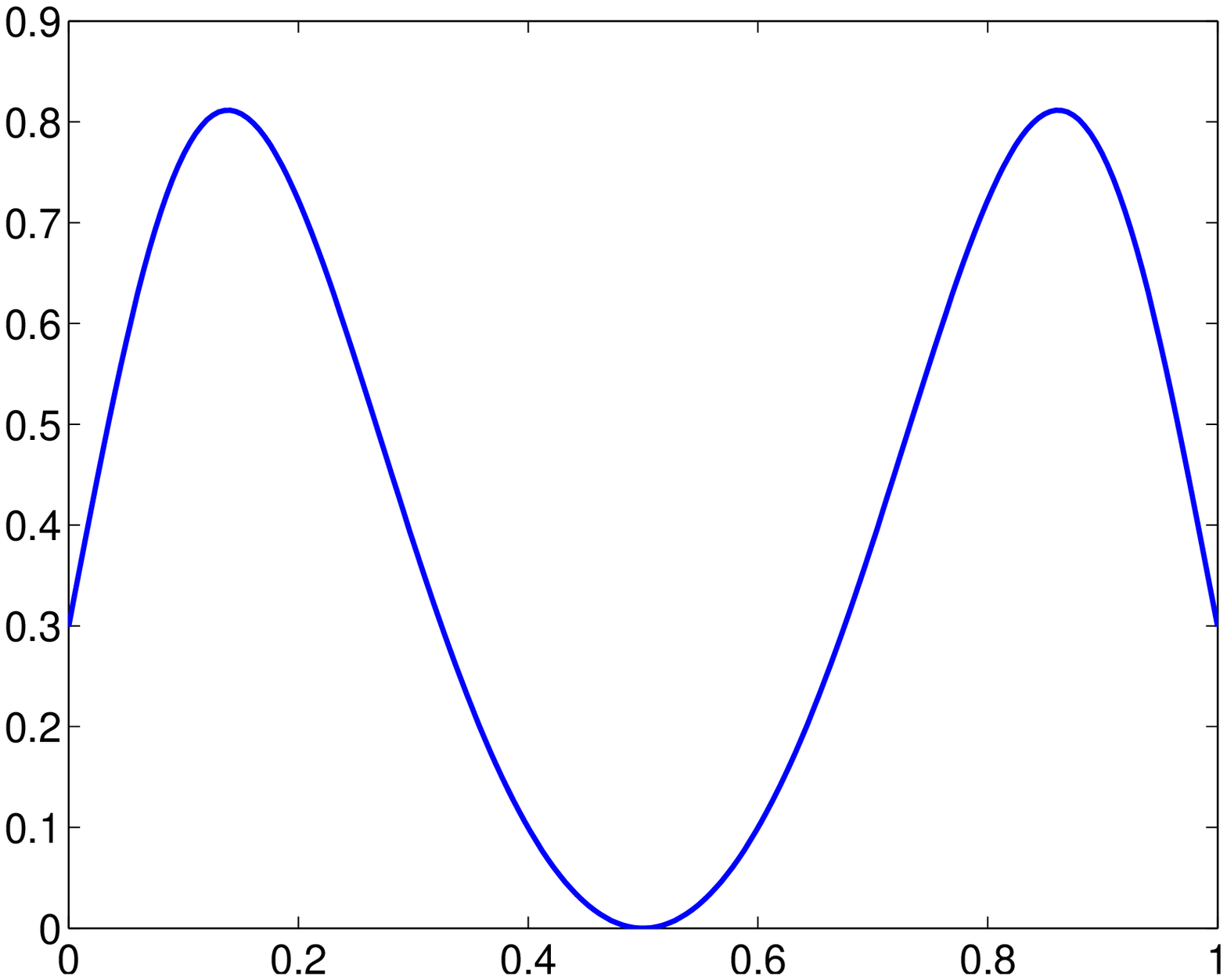}}\\
        \caption{A graphical illustration of the convolutional gridding method.  Note that $w$ is chosen for illustrative purposes only, as the corresponding decay rate of $\hat{w}$ is too slow for practical applications.}
\label{fig:cg}
\end{figure}

\begin{algorithm}
\label{alg:convgridding}
Given $\hat{f}$ at the non-equispaced measurement nodes $\lambda_k$, $k = -n,\cdots,n$. Choose truncation parameter $q$,
interpolating window function $w$, and DCFs $\alpha_k$ in (\ref{f-cg}).
\begin{enumerate}
\item Interpolate to equispaced nodes via convolution:
\begin{equation}
\label{eq:g_hat}
\hat{\tilde{g}}(l) = \sum_{k \mbox{ st. } |l-\lambda_{k}|\leq q} \alpha_k \hat{f}(\lambda_{k}) \hat{w}(l-w_{k}) \approx \hat{g}(l), \quad l = -n,\cdots,n.
\end{equation}
\item  Perform standard FFT computation:
\begin{equation}\label{eq:g}
S_n\tilde{g}(x_p) = \sum_{l=-n}^n \hat{\tilde{g}}(l) e^{2\pi ilx_p}, \quad x_p = \frac{p}{n}, \quad p = 0, ... , n-1.
\end{equation}
\item Compensate for using the interpolating window function: \\ $\displaystyle A_{cg}(f)(x_p) =  \frac{S_n\tilde{g}(x_p)}{w(x_p)}.$
\end{enumerate}
\label{alg:grid}
\end{algorithm}

Since $\tilde{g}$ is also a piecewise smooth function in $[0,1]$, its Fourier reconstruction will likewise yield the Gibbs phenomenon.  Hence the second step is typically filtered, i.e., the coefficients are multiplied by an admissible filter function $\sigma(\eta) = \sigma(\frac{|l|}{N})$.  An in depth discussion on Fourier filters can be found in \cite{tadmor2007fma}.

\subsection{The Fourier Frame Approximation}
\label{sec:fourframe}

Convolutional gridding, as it is described in Algorithm \ref{alg:grid}, is widely used in a variety of applications. A main advantage is that the process is a forward operation, that is, no inverse problem need be solved in the calculation.  Moreover, since we are using the standard Fourier expansion, the fast Fourier transform can be employed to expedite the computation. Thus the process is straightforward to apply and very efficient.  However, the error analysis of convolutional gridding is not yet fully understood, and in fact there are some instances for which the method fails to converge, \cite{Aditya09}. Below we provide a rigorous analysis of convolutional gridding by linking it to a numerical frame approximation for piecewise analytic functions $f$, \cite{Christensen2003}.  By applying the corresponding error analysis from \cite{SongGelb2013}, we are able to demonstrate the existence of DCFs and window functions to ensure the numerical convergence of \eqref{f-cg}. To this end, we first review the frame definition and the numerical frame approximation.

Let $\cH$ be a separable Hilbert space. We say $\{\varphi_j : j \in \bZ\}\subseteq \cH$ is a {\em frame} for $\cH$ if there exists some positive constants $A, B$ such that for all $f\in \cH$
\begin{equation*}
A\|f\|_\cH^2 \leq \sum_{j\in \bZ}|\langle f, \varphi_j \rangle|^2 \leq B\|f\|_\cH^2.
\end{equation*}
For a frame $\{\varphi_j : j \in \bZ\}\subseteq \cH$, its frame operator $S: \cH\rightarrow \cH$ is defined as
\begin{equation*}
S(f):= \sum_{j\in \bZ}\langle f, \varphi_j \rangle \varphi_j, \quad f\in \cH.
\end{equation*}
The standard frame expansion has the form
\begin{equation}\label{eq:frmexpansion}
f = S^{-1}Sf = \sum_{j\in \bZ}\langle f, \varphi_j \rangle \tilde{\varphi}_j
= \sum_{j\in \bZ} \hat{f}(\lambda_j) \tilde{\varphi}_j,
\end{equation}
where $\{\tilde{\varphi}_j:=S^{-1}\varphi_j: j\in \bZ\}$ is called the {\em (canonical) dual frame} and we have defined the frame coefficients as $\hat{f}(\lambda_j) := \langle f, \varphi_j \rangle$. 
In our investigation we will assume the sampling sequence $\{\varphi_j : j \in \bZ\}$ is a frame in $L^2[0,1]$.\footnote{We note that this is assumption is validated in \cite{BW00} for some non-Cartesian sampled MRI data.} 

Since typically the dual frame is not explicitly known, the (modified) frame algorithm, \cite{Christensen2003}, is often used to calculate (\ref{eq:frmexpansion}).  Other inverse frame operator approximation methods have also been developed, \cite{Casazza2000, MR1955936, Christensen2005,  Grochenig1993}.  Here we choose to use the method of {\em admissible frames}, \cite{SongGelb2013}, to obtain an approximation of $\tilde{\varphi}_j$.   This method  provides a link from (\ref{eq:frmexpansion}) to \eqref{f-cg} and therefore relates its corresponding convergence properties. 

Briefly, an admissible frame is designed to compute the inverse frame operator when the given frame may only be weakly localized (the off-diagonal decay has order less than $1$).  Specifically, we define \cite{SongGelb2013}:
\begin{defn}\label{admissible}
A frame $\{\psi_l: l\in\bZ\}$ is {\it admissible} with respect to the frame $\{\varphi_j: j\in \bZ\}$ if the following two conditions hold:
\begin{enumerate}[(i)]
\item There exist some positive constants $c_0$ and $t>1$ such that
\begin{equation*}
\left|\langle \psi_j, \psi_l\rangle \right|\leq c_0 (1+|j-l|)^{-t}, \quad j,l\in \bZ.
\end{equation*}

\item There exist some positive constants $c_1$ and $s>\frac{1}{2}$ such that
\begin{equation*}
\left|\langle \varphi_j, \psi_l\rangle \right|\leq c_1 (1+|j-l|)^{-s}, \quad j,l\in \bZ.
\end{equation*}
\end{enumerate}
\end{defn}

Now suppose $\{\psi_l: l\in\bZ\}$ is an admissible frame with respect to the frame $\{\varphi_j: j\in \bZ\}$.  As shown in \cite{SongGelb2013}, the dual frame $\tilde{\varphi}_j =S^{-1}\varphi_j$ can be approximated by
\begin{equation}\label{eq:approxdual}
\tilde{S}_n^{-1}\varphi_j:=\sum_{|l|\leq k}b_{l,j}\psi_l,
\end{equation}
where $\mB:=[b_{l,j}]_{|l|\leq k, |j|\leq n}$ is the Moore-Penrose pseudo-inverse of the matrix 
\begin{equation}\label{eq:mPsi}
\mPsi:=[\langle\varphi_j, \psi_l \rangle]_{|j|\leq n, |l|\leq k}.
\end{equation}
The relationship between $n$ and $k$ is given by
$
n=k+ c k^{\frac{1}{2s-1}},
$
where $c$ is a constant depending on $c_1$, $s$, and the frame bounds of $\{\varphi_j: j\in \bZ\}$, \cite{SongGelb2013}. 

The formula given in (\ref{eq:approxdual}) can be directly inserted into the truncated form of (\ref{eq:frmexpansion}) as an approximation of $f$, and, as shown in \cite{SongGelb2013}, the rate of convergence is of the same order as that of approximations using the new expansion frame, $\{\psi_l\}$.

As an example, consider the jittered samples (\ref{eq:jit_samp_exp}) with $\theta = \frac{1}{4}$ and its corresponding frame $\{\varphi_j(x) = \se^{2\pi i\lambda_j}:j \in \bZ\}$. The standard Fourier basis $\{\psi_l(x) = \se^{2\pi ilx}: l\in \bZ\}$ is admissible with respect to $\varphi_j$ for any positive number $t$ and $s=1$.

Substituting \eqref{eq:approxdual} into the frame expansion \eqref{eq:frmexpansion} and truncating the infinite summation of $j$ yields the following approximation of $f$:
\begin{equation}\label{f-frm}
A_{frm}(f):=  \sum_{|l|\leq m}\sum_{|j|\leq n} \hat{f}(\lambda_j)  b_{l,j} \psi_l.
\end{equation}
To link (\ref{f-frm})  to the convolutional gridding method, $A_{cg}(f)$ in \eqref{f-cg}, we take $\psi_l(x)=\frac{\se^{2\pi i lx}}{w(x)}$.

We next provide the error analysis for the frame approximation method, $A_{frm}(f)$, following the approach of admissible frames proposed in \cite{SongGelb2013}. To this end, we first show that $\{\psi_l: l \in \bZ\}$ is an admissible frame with respect to $\{\varphi_j:  j\in \bZ\}$.

\begin{prop}
\label{propRiesz}
If there exist some positive constants $\alpha_L$ and $\alpha_U$ such that $\alpha_L\leq w(x)\leq \alpha_U$ for all $x\in [0, 1]$, then $\{\psi_l(x)=\frac{\se^{2\pi i l x}}{w(x)}: l\in \bZ\}$ is a Riesz basis (and therefore also a frame, \cite{Christensen2003}) for $L^2[0,1]$ with frame bounds $1/\alpha_U^2$ and $1/\alpha_L^2$.
\end{prop}
\begin{proof}
Recall that a Riesz basis is the image of a bounded invertible mapping of an orthonormal basis \cite{Christensen2003}. Since $\{\se^{2\pi i l x}: l\in \bZ\}$ is an orthonormal basis for $L^2[0,1]$ and the bound of $w(x)$ ensures that the mapping $U(f):=f/w$ is bounded invertible in $L^2[0,1]$, $\{\psi_l(x)=U(\se^{2\pi i l x}): l\in \bZ\}$ is a Riesz basis for $L^2[0,1]$.  The frame bounds follow from a direct calculation.
\end{proof}

We next discuss the conditions for which $\{\psi_l(x): l\in \bZ\}$ is an admissible frame.
\begin{prop}\label{prop:admissible}
If the 1-periodic extension of $w(x)$ is $C^p(\bR)$ smooth for some $p\in \bN$ and $w(x)> 0$ for any $x\in [0,1]$, then for the jittered sampling $\{\lambda_j: j\in \bZ\}$ given in (\ref{eq:jit_samp_exp}),  $\{\psi_l(x)=\frac{\se^{2\pi i l x}}{w(x)}: l\in \bZ\}$ is an admissible frame with respect to $\{\varphi_j(x) = \se^{2\pi i \lambda_j x} :  j\in \bZ\}$ with $t=p+1$ and $s=1$.
\end{prop}
\begin{proof}
We will check the two admissible conditions in Definition \ref{admissible} directly.  For $j,l\in \bN$, we have
\begin{equation}
\langle\psi_j, \psi_l \rangle = \int_{0}^1 \frac{1}{w^2(x)} \se^{2\pi i (j-l) x}dx.
\label{eq:inner}
\end{equation}
Since $w\in C^p(\bR)$ and $w(x)>0$ for all $x\in [0,1]$, we have $1/w^2\in C^p(\bR)$. Note that (\ref{eq:inner}) can be viewed as the $(l - j)$th Fourier coefficient of $1/w^2$. If a function is in $C^p(\bR)$, then its Fourier coefficients decay as $p+1$, \cite{Gasquet1999, GO77}. Thus $\{\psi_l(x): l\in \bZ\}$ satisfies the first admissible condition with $t=p+1$.

To check the second admissible condition, we note that
\begin{equation*}
\langle\varphi_j, \psi_l \rangle = \int_{0}^1 \frac{1}{w(x)} \se^{2\pi i (\lambda_j -l ) x}dx.
\end{equation*}
By assumption, $1/w$ is bounded on $[0,1]$. Hence for $j,l\in \bN$, it is clear that that for some $c > 0$ we have
\begin{equation*}
|\langle\varphi_j, \psi_l \rangle|\leq c(1+|\lambda_j -l|)^{-1}.
\end{equation*}
The second admissibility condition is readily verified for $s = 1$ in the jittered sampling case, (\ref{eq:jit_samp_exp}).
\end{proof}

The convergence results for the frame approximation method $A_{frm}(f)$ can now be given. In what follows, when it is clear from the context, we will use $\|\cdot\|$ denote all norms. This notation includes the $L^2[0,1]$ norm $(\int_0^1 |f(x)|^2 dx)^{1/2}$ on a function $f$, the Euclidean norm of a vector, or the spectral norm of a matrix.

\begin{thm}
\label{thm:frameapprox}
Suppose the 1-periodic extension of $w(x)$ is $C^p(\bR)$ smooth for some $p\in \bN$ and there exists a positive constant $c_0$ such that $|\hat{f}(l)| \leq c_0 (1+|l|)^{-p-1}$ for all $l\in \bZ$. If there exist positive constants $\alpha_L$ and $\alpha_U$ such that $\alpha_L\leq w(x)\leq \alpha_U$ for any $x\in [0,1]$, then for the jittered sampling case, $\{\lambda_j: j\in \bZ\}$, with $m=(1+2c_1^2\alpha_U^4)n$ in (\ref{f-frm}), we have
\begin{equation*}
\|f-A_{frm}(f)\|\leq c n^{-p-1/2} 
\end{equation*}
for some positive constant $c>0$. 
\end{thm}
\begin{proof}
From Proposition \ref{prop:admissible}, we know $\{\psi_l(x)=\frac{\se^{2\pi i l x}}{w(x)}: l\in \bZ\}$ is an admissible frame with respect to $\{\varphi_j(x):  j\in \bZ\}$ with $t=p+1$ and $s=1$. The desired result follows immediately from Theorem 5.1 in \cite{SongGelb2013}.\footnote{Theorem 5.1 in \cite{SongGelb2013} requires the use of several technical propositions and is therefore omitted here.}
\end{proof}
Note that Theorem \ref{thm:frameapprox} does not apply to the case of log sampling, since in general its corresponding sequence does not constitute a frame.  However, as mentioned in the introduction, in \cite{BW00} it was suggested that some MRI sampling patterns may yield Fourier frames.

\subsection{Linking Convolutional Gridding and the Fourier Frame Approximation}
\label{sec:linkingmethods}
We now have two different approaches for approximating piecewise smooth $f$  in $[0,1]$ from non-uniform Fourier data -- the convolutional gridding method, $A_{cg}(f)$ in (\ref{f-cg}), and the Fourier frame approximation, $A_{frm}(f)$ in (\ref{f-frm}).  In order to link the two approximations, we first slightly modify the summation in $A_{cg}(f)$ to make the summation index consistent, and define
\begin{equation}\label{tf-cg}
\tA_{cg}(f): = \sum_{|j|\leq n}\,\, \sum_{|l|\leq m}\alpha_j \hat{f}(\lambda_j) \hat{w}(l-\lambda_j) \frac{\se^{2\pi ilx}}{w(x)}.
\end{equation}
When $\hat{w}$ decays fast or is compactly supported, the above modification should not significantly change the computational results or the convergence analysis.  

In order to compare $A_{frm}(f)$ to the (modified) convolutional gridding method $\tA_{cg}(f)$, we first rewrite both methods using matrix notation. By defining the vector $\vhf:=(\hat{f}(\lambda_j): |j|\leq n)$, (\ref{tf-cg}) yields
\begin{equation}\label{eq:Acg}
\tA_{cg}(f) = \sum_{|l|\leq m}c_l\psi_l, \quad \vc = (c_l: -m\leq l\leq m)= \mOmega \mLambda \vhf,
\end{equation}
where $\mOmega=[\hat{w}(l-\lambda_j)]_{|l|\leq m, |j|\leq n}$ and $\mLambda = \diag(\alpha_j: |j|\leq n)$. Likewise, we rewrite $A_{frm}$ in \eqref{f-frm} as
\begin{equation}\label{eq:Afrm}
A_{frm}(f) = \sum_{|l|\leq m}d_l\psi_l, \quad \vd =  (d_l: -m\leq l\leq m) = \mB \vhf = \mPsi^{\dagger}\vhf,
\end{equation}
where $\mPsi^{\dagger}$ is the Moore-Penrose pseudo-inverse of $\mPsi$ in \eqref{eq:mPsi}.

Theorem \ref{thm:frameapprox} provides convergence results for $A_{frm}(f)$.  We are now able to establish the difference between the  approximations  $\tA_{cg}(f)$ and $A_{frm}(f)$.
\begin{thm}
\label{thm:link}
Suppose there exist some positive constants $\alpha_L$ and $\alpha_U$ such that $\alpha_L\leq w(x)\leq \alpha_U$ for all $x\in [0, 1]$. Then there exists a positive constant $c$ such that
\begin{equation*}
\|\tA_{cg}(f) - A_{frm}(f)\| \leq  \frac{1}{\alpha_L} \|(\mPsi^* \mPsi)^{-1}\| \|\mPsi^* \mPsi \Omega \mLambda -\mPsi^* \|_{F} \|\vhf\|,
\end{equation*}
where $\|\cdot\|_F$ is the Frobenius norm and $\mPsi^*$ is the conjugate transpose of $\mPsi$.
\end{thm}
\begin{proof}
From (\ref{eq:Acg}) and (\ref{eq:Afrm}) we have
\begin{equation*}
\|\tA_{cg}(f) - A_{frm}(f)\|=\biggl\|\sum_{|l|\leq m}(c_l -d_l) \psi_l \biggr\| = \biggl\|\sum_{|l|\leq m}(c_l -d_l) \se^{2\pi i lx}\frac{1}{w(x)} \biggr\|.
\end{equation*}
It follows that
\begin{equation*}
\|\tA_{cg}(f) - A_{frm}(f)\|\leq \|1/w \|_{\infty}  \left\|\sum_{|l|\leq m}(c_l -d_l) \se^{2\pi i lx} \right\| \leq \frac{1}{\alpha_L} \|\vc- \vd\|.
\end{equation*}
To estimate $\|\vc- \vd\|$, we observe that
\begin{equation*}
\vc -\vd = ( \mOmega \mLambda - \mPsi^{\dagger})\vhf = \left( \mOmega \mLambda -(\mPsi^* \mPsi)^{-1} \mPsi^* \right)\vhf =  (\mPsi^* \mPsi)^{-1} (\mPsi^* \mPsi \Omega \mLambda -\mPsi^* ) \vhf.
\end{equation*}
Therefore
\begin{equation*}
\|\vc- \vd\|\leq \|(\mPsi^* \mPsi)^{-1}\| \|\mPsi^* \mPsi \Omega \mLambda -\mPsi^* \|_{F} \|\vhf\|,
\end{equation*}
which finishes the proof.
\end{proof}
\begin{rem}
Theorem \ref{thm:link} depends upon the representation of the dual frame given by \eqref{eq:approxdual}.  Various versions of the frame algorithm in \cite{Christensen2003} do not yield the form needed in (\ref{eq:Afrm}).
\end{rem}
\begin{rem}
Theorem \ref{thm:link} actually suggests a way to select optimal DCFs in (\ref{eq:ghatl}).  Specifically, for any given sampling sequence $\{\lambda_j\}_{j = -n}^n$ and weight function $w$, the matrices $\mPsi$ and $\mOmega$ are both fixed.  Hence the optimal $\mLambda$ satisfies
\begin{equation}\label{eq:optimalDCF}
\mathop{\argmin}\limits_{\valpha\in \bR^{2n+1}}\|\mPsi^* \mPsi \Omega \mLambda -\mPsi^* \|_{F},
\end{equation}
where $\|\cdot\|_F$ is the Frobenius norm.  The DCFs, $\{\alpha_j\}_{j = -n}^n,$ are the diagonal elements in the diagonal matrix $\mLambda$. In fact, we have
\begin{lem}
Let $D=\diag (\alpha_j: -n\leq j\leq n)$, where $\alpha_j$, $-n \leq j \leq n$ are the  DCFs in \eqref{eq:optimalDCF}. Then the optimal DCFs $\alpha^\prime_j$ are given by
\begin{equation*}
\alpha^\prime_j = \frac{\langle \mPsi^*_j, \mT_j\rangle}{\|\mT_j\|^2},
\end{equation*}
where $\mPsi^*_j$ and $\mT_j$ are the $j$-th column of $\mPsi^* $ and $\mPsi^* \mPsi \Omega$ respectively.  
\end{lem}
\begin{proof}
The result follows from calculating
\begin{equation*}
\|\mPsi^* \mPsi \Omega \mLambda -\mPsi^* \|_{F}^2 = \sum_{j=-n}^n \|\alpha_j\mT_j - \mPsi^*_j \|^2
\end{equation*}
as the solution to (\ref{eq:optimalDCF}).
\end{proof}
\end{rem}

\begin{rem}
By defining $\alpha_j$ as the diagonal elements of $\mLambda$, Theorem \ref{thm:link} allows us to exploit the frame approximation to obtain the optimal DCFs as they appear in the traditional convolutional gridding algorithm.  However, the results can be improved by allowing $\mLambda$ to be banded instead of diagonal. In particular, in this case we view the DCFs as linear combinations of weights used in approximating the finite frame approximation of $f$ from its sampled non-uniform Fourier data, as opposed to viewing them as numerical quadrature weights that approximate the inverse Fourier transform.  For example, suppose $\mLambda$ is a $(2r-1)$-banded matrix with $r-1$ sub-diagonals above and below the main diagonal respectively.  The optimal banded matrix is determined by solving
\begin{equation}\label{eq:gcg}
\argmin\left\{\|\mPsi^* \mPsi \Omega \mLambda -\mPsi^* \|_{F}: \mLambda \mbox{ is $(2r-1)$-banded} \right\}.
\end{equation}

If we define $\mT:=\mPsi^* \mPsi \mOmega$,  it follows that
\begin{equation*}
\|\mT \mLambda -\mPsi^*\|_{F}^2 = \sum_{j=1}^n \|\mT \md_j - \mPsi^*_j\|^2,
\end{equation*}
where $\md_j$ are the column vectors of $\mLambda$.  The solution to \eqref{eq:gcg} is also the minimizer of $\|\mT \md_j - \mPsi^*_j\|^2$ for each $j$. This is a standard least-squares problem with the explicit solution $\md_j=\mT^\dagger \mPsi^*_j$. Note that the number of non-zero components in each $\md_j$ is at most $2r-1$. That is, we actually need only to calculate the Moore-Penrose pseudo-inverse of the corresponding block (at most $(2r-1)\times (2r-1)$) instead of the whole matrix $\mT$.

The computational cost of calculating $\mLambda$ is $O(nr^2)$.  Once $\mLambda$ is determined, the cost of evaluating $f$ using the FFT is $\mathcal{O}(n\log n)$. Note that the computational cost of determining the DCFs would still be consistent with that of the FFT if we choose $r=\mathcal{O}(\log n)$. In fact, we observe in some of our experiments in Section \ref{sec:numerics} that numerical convergence, especially in the log sampling case, requires $r$ to grow with $n$.  Indeed, the method may not converge for constant $r$.
\end{rem}

Figures \ref{fig:dcf:jit} and \ref{fig:dcf:log}  compare the DCFs generated by (A) the trapezoidal rule, (\ref{eq:simple_dcf}),  (B) the iterative procedure proposed in \cite{pipe}, and (C) our new frame theoretic convolutional gridding approach for both the jittered (\ref{eq:jit_samp_exp}) with $\theta = \frac{1}{4}$ and log sampling cases. As discussed in Section \ref{sec:convgridding}, iterative techniques typically yield more suitable  DCFs than the quadrature weights, such as the trapezoidal rule. As is evident here, our new method generates comparable DCFs as those constructed by iterative procedures, although it has less computational complexity.  Note that although the corresponding sequence for the log sampling scheme does not constitute a frame, the same algorithm used to generate the optimal (frame theoretic) DCFs can still be applied.  In Section \ref{sec:numerics}, we will also demonstrate that the numerical approximation using the frame theoretic convolutional gridding method for log sampling sequences also yield high accuracy for smooth functions.
\begin{figure}[htbp]
        \subfloat[Trapezoidal]{\includegraphics[scale = .25]{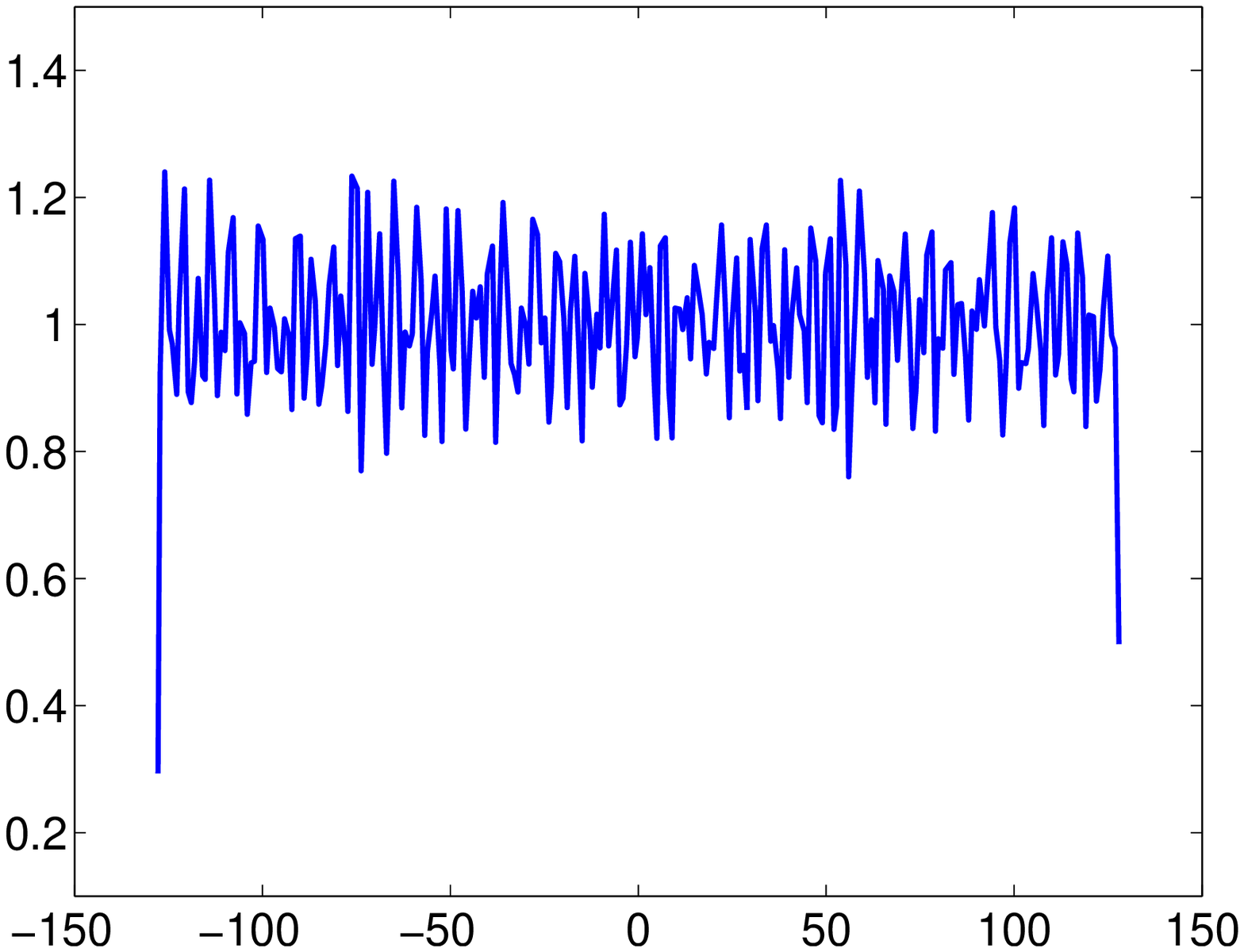}}
        \subfloat[Iterative]{\includegraphics[scale = .25]{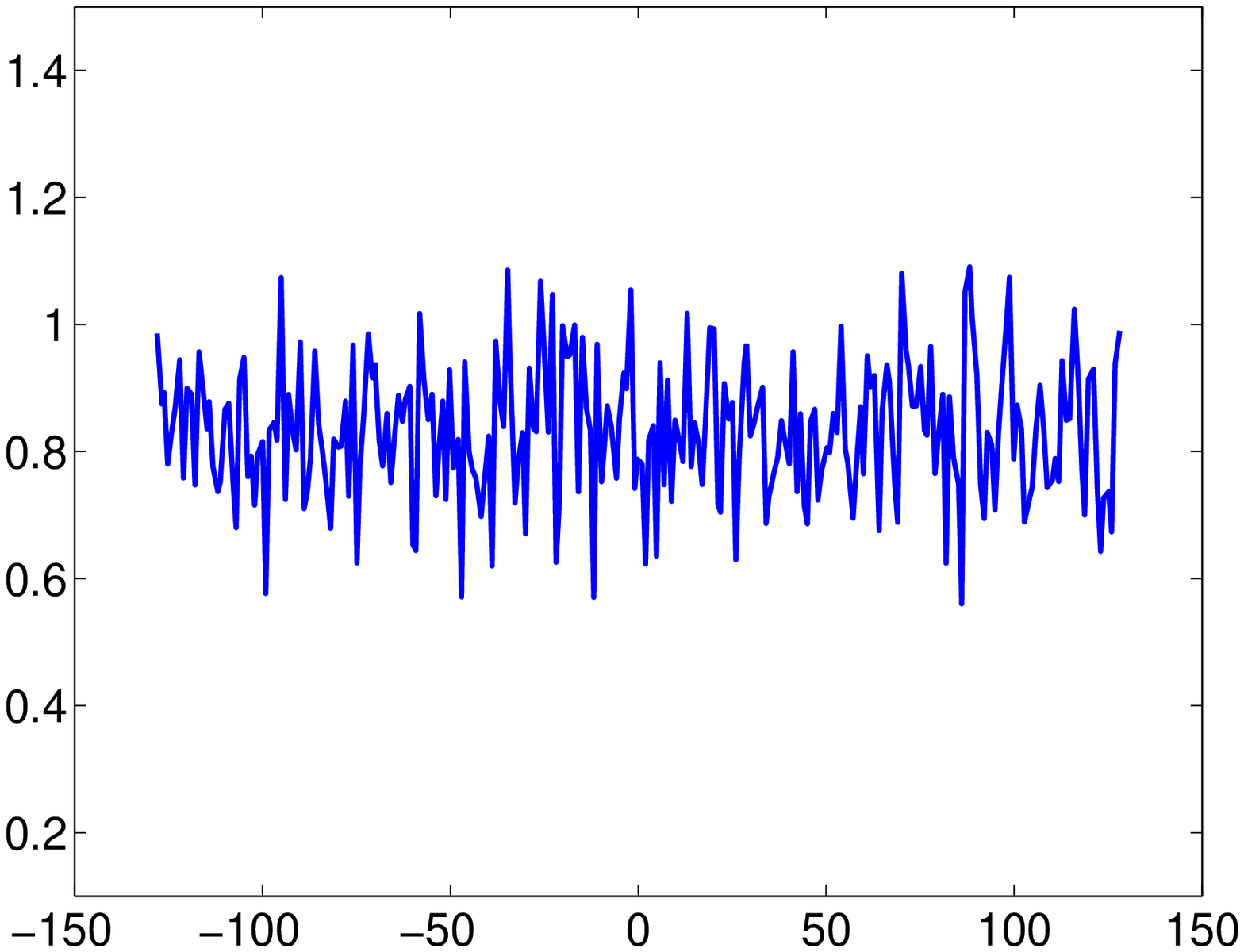}}
        \subfloat[Frame Theoretic]{\includegraphics[scale = .25]{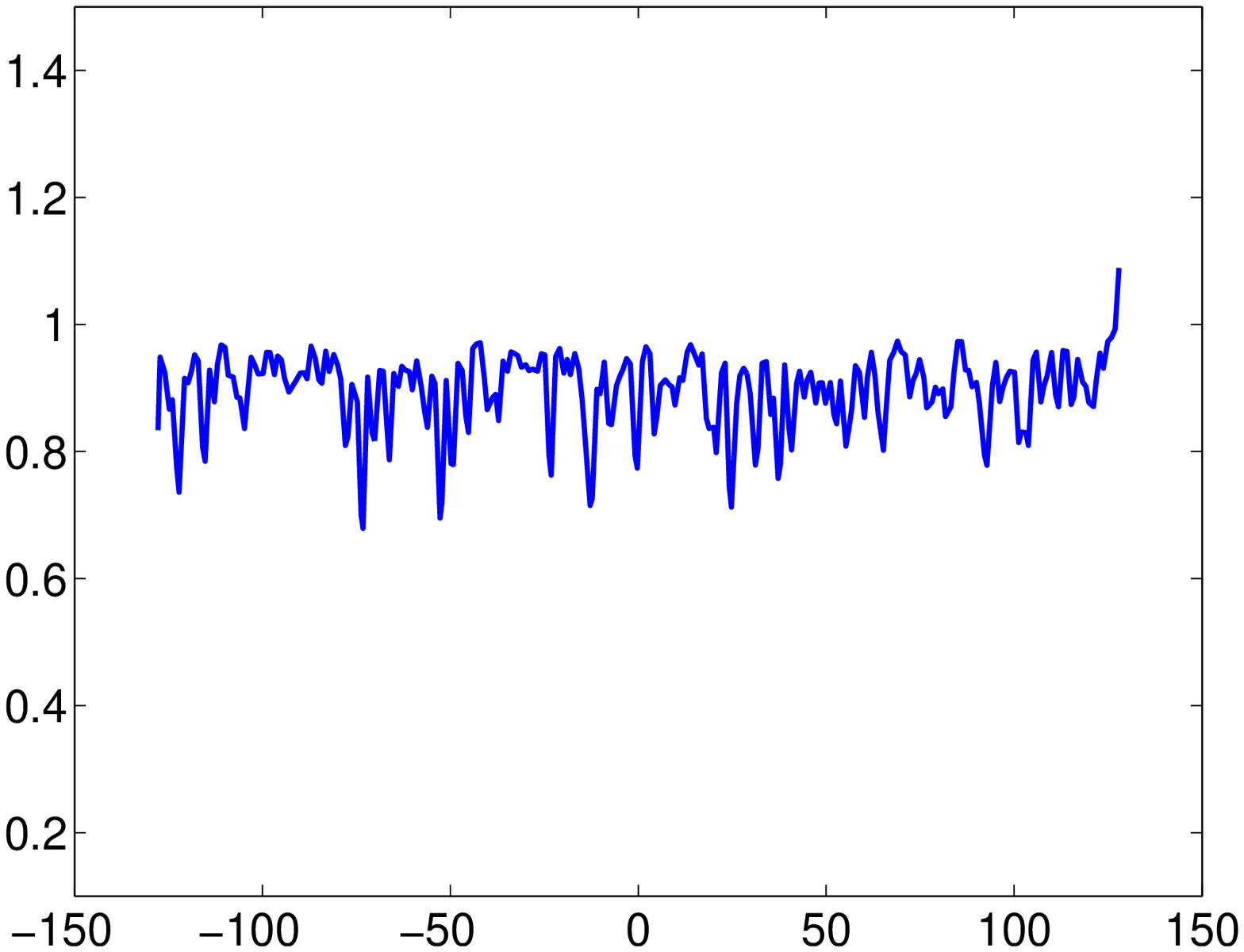}}
        \caption{DCFs for jittered sampling with $n=128$. We used the window function $w(x) = \se^{-a|x-.5|}$ where $a = 5\times 10^{-5}$ for (B) and (C).}
\label{fig:dcf:jit}
\end{figure}

\begin{figure}[htbp]
        \subfloat[Trapezoidal]{\includegraphics[scale = .25]{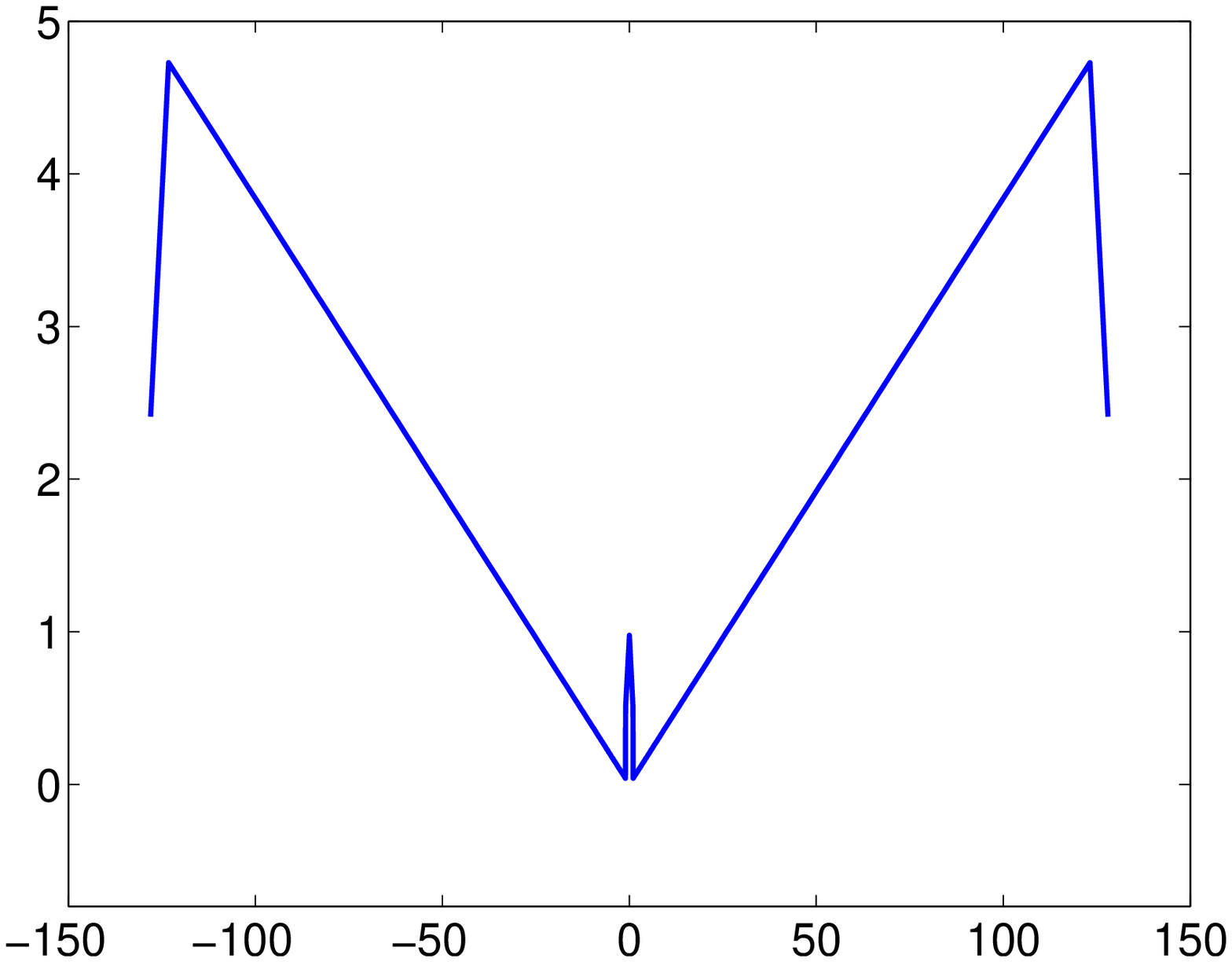}}
        \subfloat[Iterative]{\includegraphics[scale = .25]{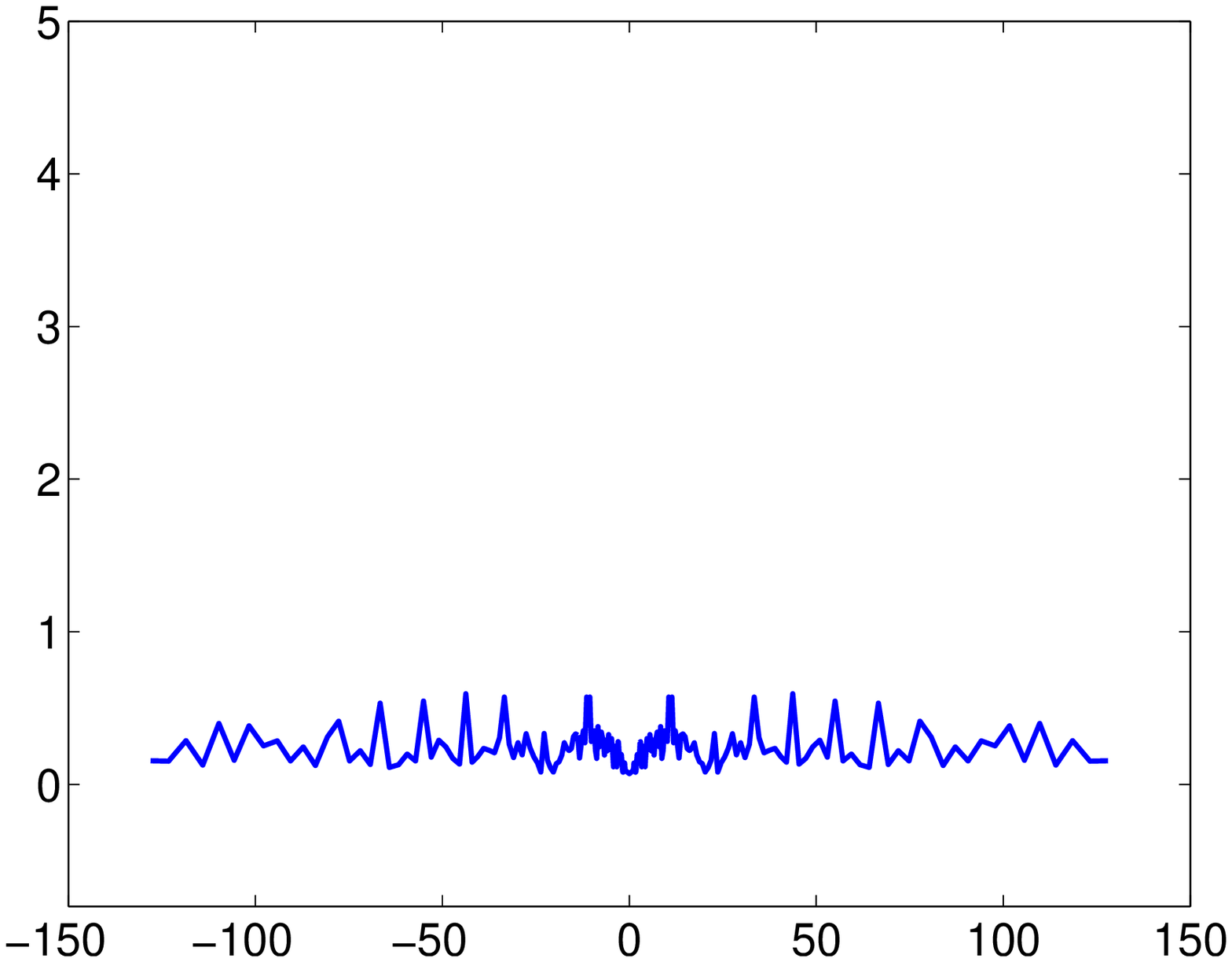}}
        \subfloat[Frame Theoretic]{\includegraphics[scale = .25]{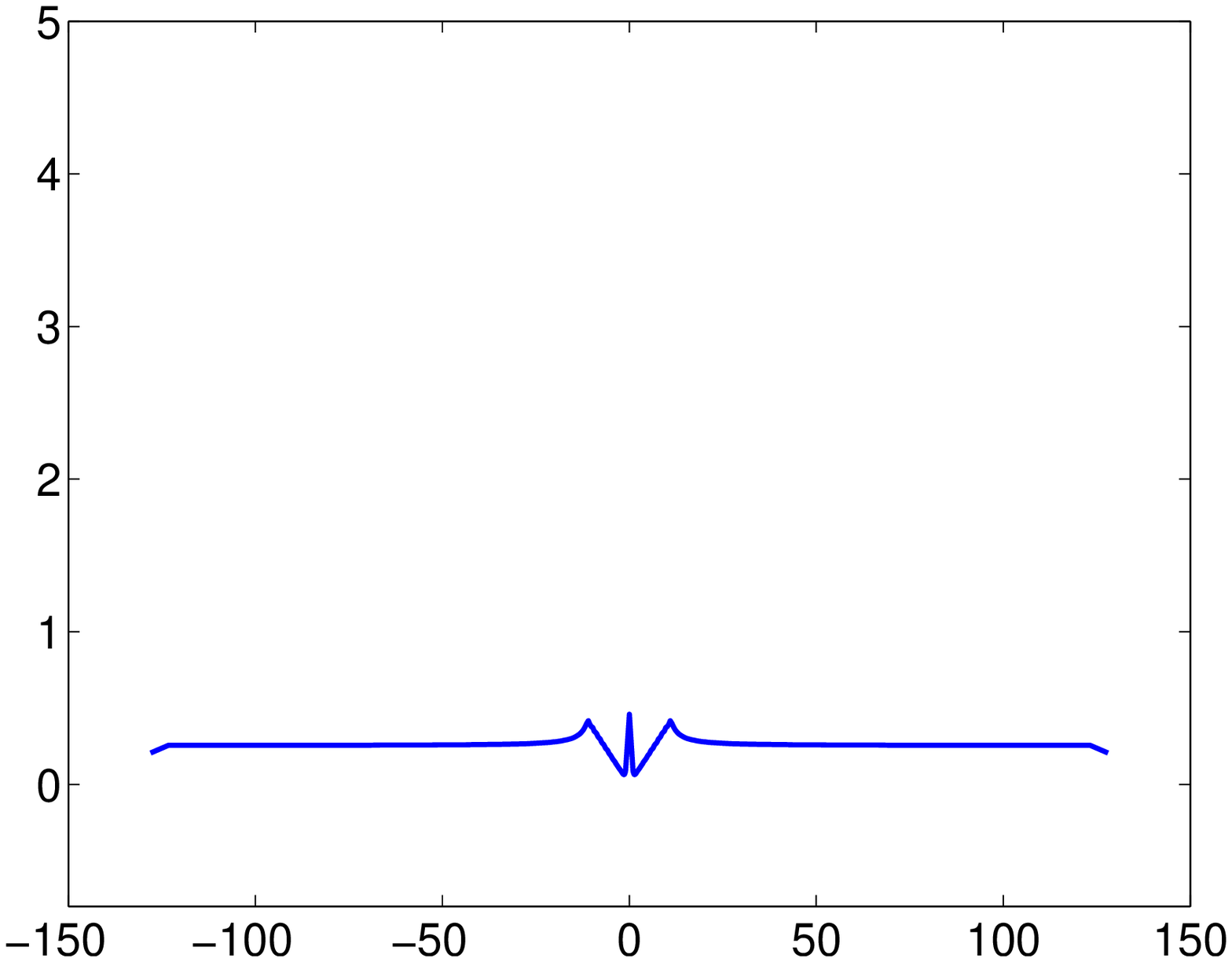}}
        \caption{DCFs for log sampling with $n=128$. We used the window function $w(x) = \se^{-a|x-.5|}$ where $a = 5\times 10^{-5}$ for (B) and (C).}
\label{fig:dcf:log}
\end{figure}

\begin{rem}
Figure \ref{fig:dcf:jit} suggests that interpreting the reconstruction of $f$ as an approximation to its inverse Fourier transform and subsequently applying a quadrature formula (of general order) is suitable when the data are close to uniform.  On the other hand, when the given data are highly non-uniform, Figure \ref{fig:dcf:log} shows that standard quadrature rules ``compensate'' by {\em increasing} the DCF values in the sparse regions.  Unfortunately, in many signal processing applications the high frequency data are likely to contain noise  which should be mollified rather than enhanced.   Both the iterative and frame theoretic approaches yield more evenly weighted DCFs.  We point out that the iterative approach involves solving a full $(2n-1)\times (2n-1)$ linear system,\footnote{In some sense, the iterative DCF approach in \cite{pipe} attempts to construct a kernel from the non-harmonic series expansion that resembles a regularized Dirac delta function.  This is not the main focus of our current investigation, but will be analyzed in more detail in future work.} while the frame theoretic approach only solves $2n-1$ much smaller $(2r-1)\times (2r-1)$ linear systems. In particular we use the diagonal matrix, that is, $r = 1$, to generate Figures \ref{fig:dcf:jit} and \ref{fig:dcf:log} for the frame theoretic approach.
\end{rem}

\section{Applications to Edge Detection}
\label{sec:edgedetection}
Edge detection is important in many signal processing and imaging applications.  For example, in MRI, tissue classification inherently depends on the ability to determine the internal boundaries of an image.  In what follows we show how our new frame theoretic convolutional gridding algorithm can be applied to locate jump discontinuities in piecewise smooth functions when data are sampled non-uniformly in the Fourier domain.\footnote{We describe only the one-dimensional case, where an edge is equivalent to a jump discontinuity.  Our technique may be adapted to two dimensions if we write (\ref{eq:bump}) as a separable function. Results in this direction can be found in \cite{MGG}.}   Determining jump locations from uniformly sampled Fourier data is in itself a difficult problem, mainly because an edge is a local feature while the Fourier data are of a global nature.  While there are a variety of algorithms that have been developed to do this,  they are in general not easily adapted to the non-uniform case.  In particular, any kind of interpolation (regridding) procedure produces oscillations that do not have extractable information.  A notable exception can be found in \cite{SVGR}, where the the method from \cite{GT00} was adapted successfully for non-uniform sampling by employing a sparsity enforcing minimization procedure. However, analyzing the corresponding error is difficult, and moreover, the technique relies on solving an inverse problem which may be computationally intensive or even fail to converge.

A frame theoretic approach to edge detection from non-uniform Fourier data was first suggested in \cite{GelbHines2}.  That method relies on the (modified) frame algorithm, e.g. \cite{Christensen2003},  to compute the dual frame, which may also be computationally intensive or ill-conditioned.  Here we extend the ideas from \cite{GelbHines2}, but replace the frame algorithm  with the more efficient convolutional gridding algorithm.  The link between the Fourier frame approximation and convolutional gridding allows us to do this. 

Consider a piecewise continuous function $f: [0,1]\rightarrow \bR$ with $f(0) = f(1) = 0$. We will define
the associated {\em jump function} $[f]:(0,1) \to \bR$ by

\begin{equation}\label{eq:jumpfunction}
[f](x) = \lim_{t \downarrow x}f(t) - \lim_{t \uparrow x}f(t).
\end{equation}
Notice that $[f]$ is well-defined since $f$ is piecewise continuous. Furthermore, $[f] = 0$ everywhere except at the jump discontinuity locations, where $[f]$ takes on the value of the jump.  For simplicity of presentation, we assume that $f$ has a single jump discontinuity (edge) $\xi\in (0,1)$. In our numerical examples we will demonstrate that the method works to recover multiple jump locations, that is wherever $[f](x) \ne 0$.

Let us again assume that $\{\varphi_j(x)=\se^{2\pi i \lambda_j x}: j\in \bZ\}$ is a frame for $L^2[0,1]$.   We seek the jump location, $\xi$, and its corresponding jump value, $[f](\xi)$, from the given data $\hvf:= \{\hf(\lambda_j)=\langle f, \varphi_j \rangle: -n\leq j\leq n\}$.  In constructing a numerical algorithm to determine the jump discontinuities of $f$, it will be more convenient to write (\ref{eq:jumpfunction}) as
\begin{equation}\label{eq:jumpfunction2}
[f](x) = [f](\xi)\delta_\xi(x),
\end{equation}
where $\delta_\xi(x)$ is the indicator function taking the value $1$ when $x=\xi$ and $0$ otherwise.  For multiple jumps, (\ref{eq:jumpfunction2}) is correspondingly $[f](x) = \sum_{j= 1}^J [f](\xi_j)\delta_{\xi_j}(x),$ where $J$ is the number of jump discontinuities in $f$.

Since the indicator function is nontrivial only at a single point whose set has zero measure, it is not practical to directly construct a meaningful approximation in $L^2[0,1]$. Therefore, we regularize (\ref{eq:jumpfunction2}) and instead seek to approximate
\begin{equation}\label{eq:jumpfunction_reg}
[f](x) \approx  [f](\xi)h_n(x),
\end{equation}
where
\begin{equation}\label{eq:bump}
h_n(x):=h\bigl(\frac{x-\xi}{\epsilon_n}\bigr), \epsilon_n > 0,
\end{equation}
with $h(0)=1$.  We choose $h$ to be a compactly supported bump function that approximates $\delta_\xi(x)$ such that $h_n$ is supported in $[\xi -\epsilon_n, \xi+\epsilon_n]$ with $h_{n}(\xi) =1$. As $\epsilon_n \downarrow 0$, $h_{n}(x) \rightarrow \delta_\xi(x)$.  Numerical considerations dictate that $h$ is smooth (e.g. bell-shaped) around $0$ and $\epsilon_n\rightarrow 0$ as $n\rightarrow \infty$. We remark that the choice of $\epsilon_n$  is critical in how the method performs.  As $\epsilon_n$ increases, the approximation is more regularized, but the edges are not as well localized. This trade off can be decided on a case by case basis, depending on other external influences, such as the amount of and corruption in the data, see e.g. \cite{cochran2013edge, GelbHines2}. 

We now implement the ideas from the frame theoretic convolutional gridding algorithm to construct an approximation to (\ref{eq:jumpfunction_reg}) from the given finite information $\hvf=\{\hf(\lambda_j): |j|\leq n\}$.  To this end, we first establish a relationship between the Fourier coefficients of the regularized indicator function, $\widehat{h_{n}}(\lambda_j)$, and the given data $\hf(\lambda_j)$.   For $\lambda_j \ne 0$, integration by parts on the given Fourier samples (\ref{eq:foursamples}) yields
\begin{eqnarray}
\hf_j&=&\frac{1}{2\pi i \lambda_j}[f](\xi) \se^{-2\pi i \lambda_j \xi} + \frac{1}{2\pi i \lambda_j}\int_0^1 f^\prime(x) \se^{-2\pi i \lambda_j x}dx, \quad -n\leq j\leq n\nonumber\\
&=&\frac{1}{2\pi i \lambda_j}[f](\xi) \se^{-2\pi i \lambda_j \xi} + \mathcal{O}\bigl(\frac{1}{\lambda_j^2}\bigr).
\label{eq:integration}
\end{eqnarray}
Thus, for large $\lambda_j$, we have
\begin{equation}\label{equation:fxi}
2\pi i \lambda_j \hf_j \se^{2\pi i \lambda_j \xi} \approx  [f](\xi).
\end{equation}
On the other hand, direct calculation of the Fourier coefficients in (\ref{eq:jumpfunction_reg}) gives
\begin{equation*}
\widehat{h_{n}}(\lambda_j) = \int_{-\infty}^\infty h\left( \frac{x-\xi}{\epsilon_n}\right) \se^{-2\pi i \lambda_j x}dx = 
\epsilon_n \se^{-2\pi i \lambda_j \xi}\int_{-\infty}^\infty h(u) \se^{-2\pi i \lambda_j u \epsilon_n}dx = \epsilon_n \se^{-2\pi i \lambda_j \xi} \hat{h}(\lambda_j \epsilon_n).
\end{equation*}
Finally, from  \eqref{equation:fxi} we have
\begin{equation}\label{eq:hnapprox}
[f](\xi)\widehat{h_{n}}(\lambda_j) \approx 2\pi i \lambda_j \epsilon_n \hf_j  \hat{h}(\lambda_j \epsilon_n),
\end{equation}
implying that we can obtain an approximation of $[f](\xi)\widehat{h_{n}}(\lambda_j)$ directly from the given data $\hf_j$.

We now formulate our edge detection algorithm as follows: Suppose we are given 
\begin{equation*}
\vhh:=(2\pi i \lambda_j \epsilon_n \hf_j  \hat{h}(\lambda_j \epsilon_n): |j|\leq n ),
\end{equation*}
i.e., an approximation to the non-uniform Fourier coefficients of the regularized jump function (\ref{eq:jumpfunction_reg}).  Following the frame theoretic convolutional gridding approach, we use (\ref{eq:Acg}) to obtain the approximation of $[f](\xi)h_n(x)$ as 
\begin{equation*}
E_{cg}(f): = \sum_{|l|\leq m}c_l\psi_l, \quad \vc = \mOmega \mLambda \vhh.
\end{equation*}
Likewise, the frame approximation method (\ref{eq:Afrm}) yields an approximation to  $[f](\xi)h_n$ as
\begin{equation*}
E_{frm}(f): = \sum_{|l|\leq m}d_l\psi_l, \quad \vd = \mPsi^{\dagger}\vhh.
\end{equation*}

Analogous results to Theorem \ref{thm:link} hold here, although the overall accuracy for recovering the edges strongly depends on the the error incurred in approximating (\ref{eq:integration}) by (\ref{equation:fxi}), which clearly increases for multiple jumps.  Once again the DCFs can be optimally chosen using (\ref{eq:gcg}).  We demonstrate our results in Section \ref{sec:numerics}.

\begin{rem}
The frame theoretic convolutional gridding approach for edge detection given non-uniform Fourier data has two distinct advantages.  First, it does not require interpolation to uniform nodes.  It is evident from (\ref{eq:integration}) that edge information comes from the high frequency part of the Fourier domain.  Unfortunately, the sampling patterns in applications such as MRI are typically sparse in the high frequency domain, and the algorithms that interpolate these data to integer coefficients are neither accurate nor robust, \cite{Aditya09}.  By contrast, the proposed frame theoretic convolutional gridding approach yields much smaller errors in the high frequency domain.  Second,  the frame theoretic convolutional gridding method is a forward operation that can be efficiently and accurately implemented using the FFT. This is in contrast to the frame based edge detection method introduced in  \cite{GelbHines2}, which requires approximating an inverse frame operator.
\end{rem}

\section{Numerical experiments}
\label{sec:numerics}
The examples below are used to demonstrate the frame theoretic convolutional gridding approach for function reconstruction and edge detection.  In each case we compare our new method to the traditional convolutional gridding method, $A_{cg}$ in (\ref{f-cg}), and the frame approximation method, $A_{frm}$ in (\ref{f-frm}).\footnote{There are no discernible discrepancies when $A_{cg}$ in (\ref{f-cg}) replaces the sum that is analyzed, $\tilde{A}_{cg}$ in  (\ref{tf-cg}).}  
Before presenting our numerical experiments, it is important to analyze the relationship between these three methods. The traditional convolutional gridding method shares the same form as the frame theoretic convolutional gridding approach with $r=1$, that is, when we use diagonal approximation in \eqref{eq:gcg}.  Both are linear methods, with the frame theoretic convolutional gridding construction providing a new algorithm for choosing optimal DCFs.  When $r > 1$, the frame theoretic convolutional gridding method can be viewed as a generalization of the traditional convolutional gridding approach since it allows the  use of banded rather than diagonal matrices.   Note that since the frame theoretic convolutional gridding approach is an approximation of the frame approximation, \eqref{f-frm}, the best convergence we can hope to achieve is that obtained by \eqref{f-frm}.   However, the standard frame approximation requires solving a $(2n-1)\times (2n-1)$ linear system, while the proposed frame theoretic convolutional gridding approach  must solve $2n-1$ smaller $(2r-1)\times (2r-1)$ linear systems.  As $r$ increases, we get closer to the frame approximation in \eqref{f-frm}.  Thus we see the expected trade off in computational efficiency and numerical convergence.  

For convenience, in all of our experiments  we use the window function
\begin{equation}\label{eq:windowfunction}
w(x)=\se^{-a|x - 0.5|},
\end{equation} 
where $a=5\times 10^{-5}$.  We point out that (\ref{eq:windowfunction}) is only one such as example that can be used in Theorem \ref{thm:link}, and that the choice of window function may influence the numerical approximation, a topic that will be explored in future work.

\begin{example}\label{ex:example1}
{\rm Consider the target function}
\begin{equation*}
f(x)=\cos^2(\pi(x-0.5)^2) \sin(10(x-0.5)^2), \quad x\in [0,1].
\end{equation*}
\end{example}

\begin{figure}[htbp]
\subfloat[Recovered functions]{\includegraphics[scale=0.35]{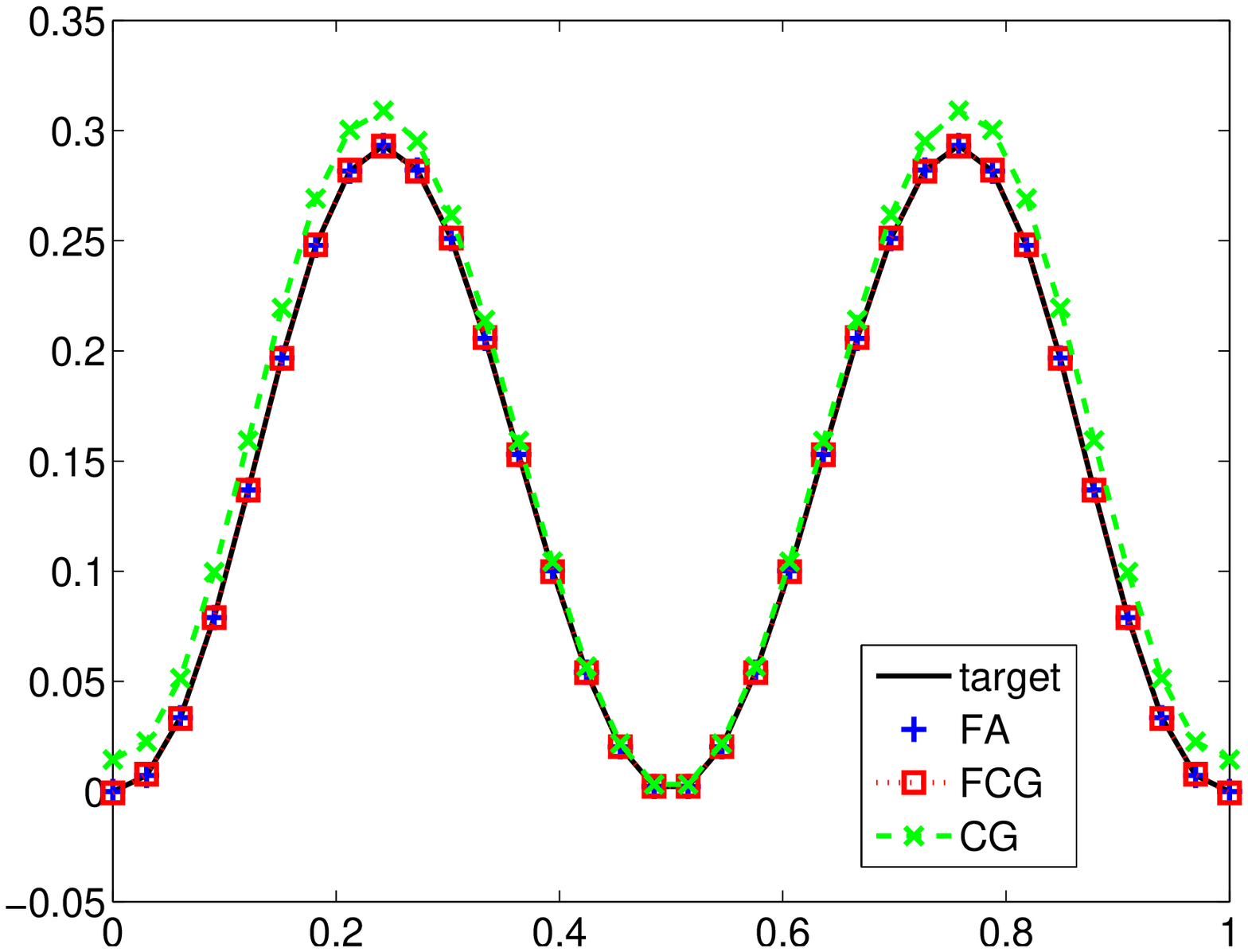}}
\qquad
\subfloat[Errors to the true function]{\includegraphics[scale=0.35]{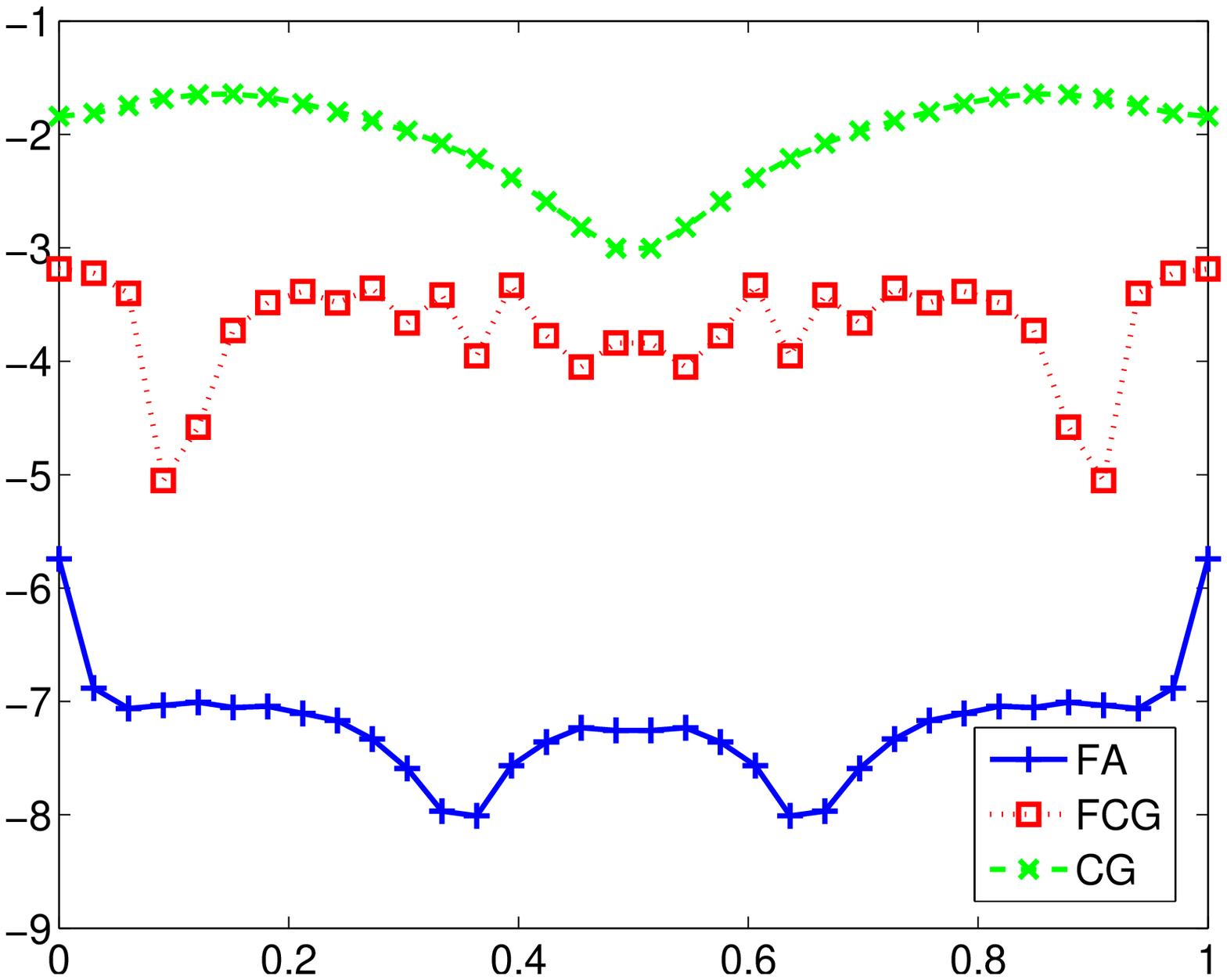}}
\caption{Comparisons of recovered functions (A) and pointwise errors (in log) (B) employing the frame approximation method (FA), the traditional convolutional gridding method (CG), and the frame theoretic convolutional gridding (FCG) for jittered sampling with $\theta = \frac{1}{4}$ and $n = 200$.}\label{fig1}
\end{figure}

Figure \ref{fig1} displays the approximation results for jittered sampling, (\ref{eq:jit_samp_exp}), with $\theta = \frac{1}{4}$.  We compare
the frame approximation method (FA), \eqref{f-frm}, the convolutional gridding method (CG), \eqref{f-cg}, and the new frame theoretic convolutional gridding (FCG), \eqref{eq:Acg}, for $r=3$. For simplicity we employ the trapezoidal rule, (\ref{eq:simple_dcf}), to determine the DCFs for the traditional convolutional gridding method (CG). 

\begin{figure}[htbp]
\subfloat[Recovered functions]{\includegraphics[scale=0.35]{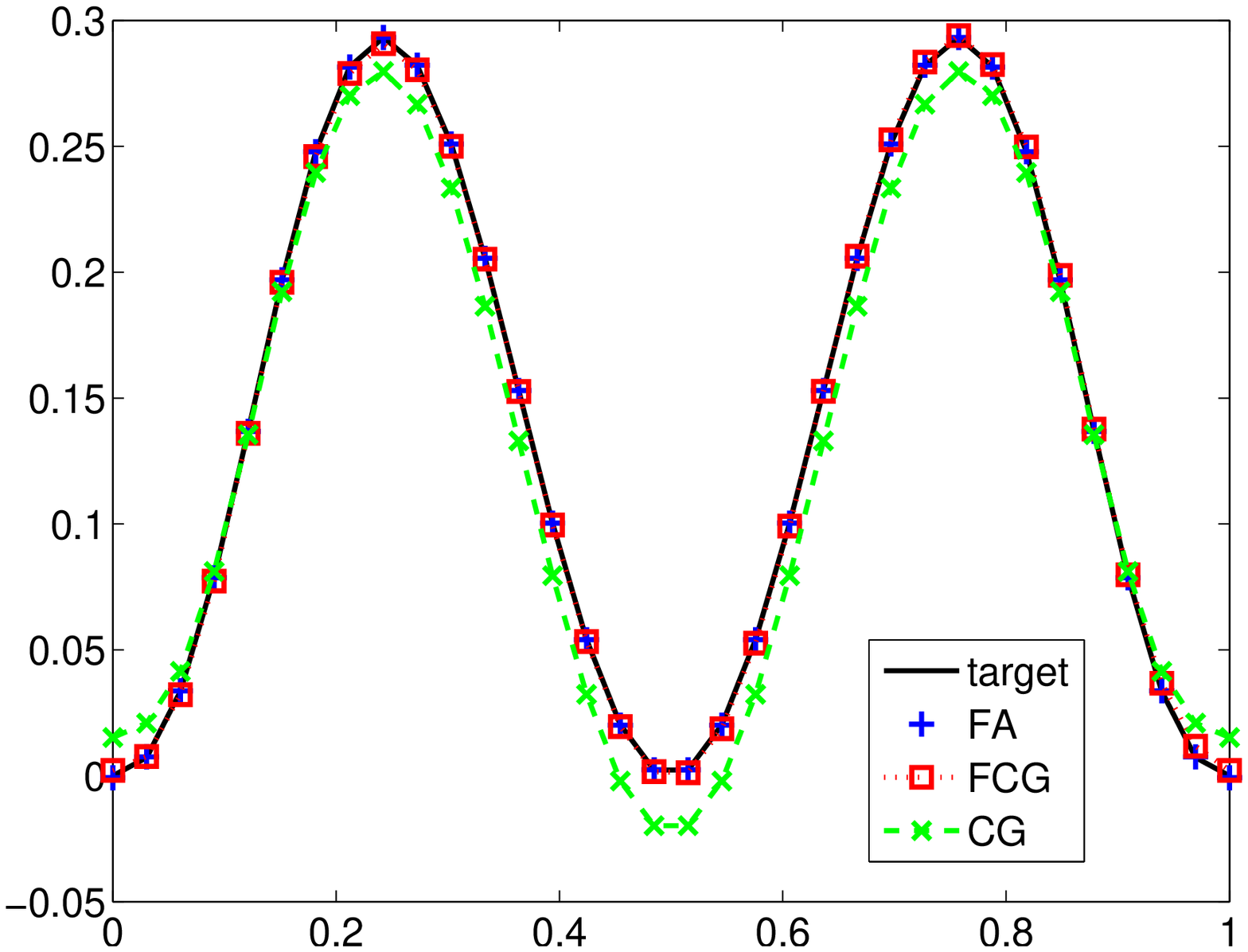}}
\qquad
\subfloat[Errors to the true function]{\includegraphics[scale=0.35]{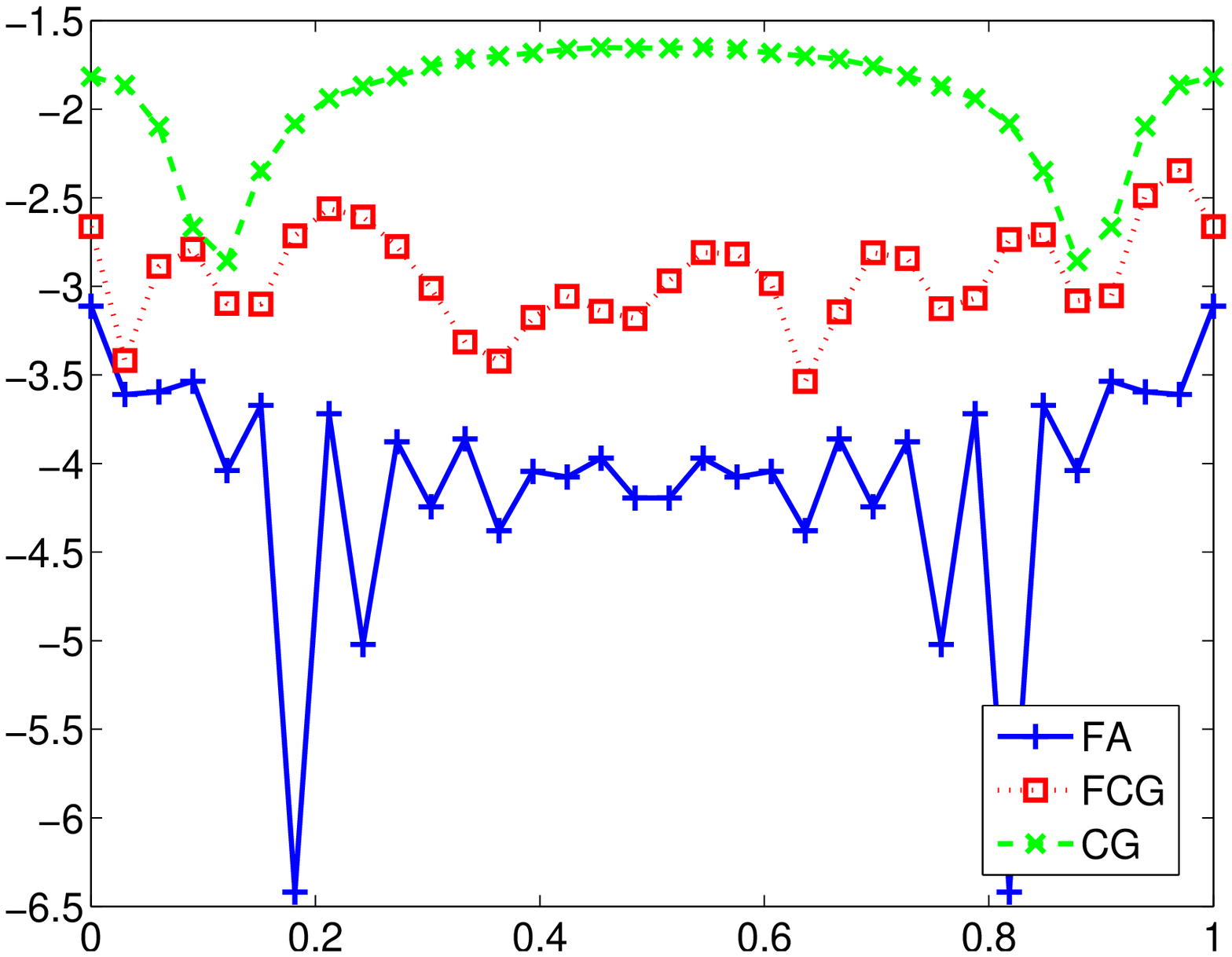}}
\caption{Comparisons of recovered functions (A) and pointwise errors (in log) (B) of these three methods: the frame approximation method (FA), the traditional convolutional gridding method (CG), and the frame theoretic convolutional gridding (FCG) for log sampling.}\label{fig1b}
\end{figure}

Figure \ref{fig1b} displays the analogous results for the log sampling pattern. 
Note in this case that the corresponding sequence $\{\varphi_j : j \in \bZ\}$ may not constitute a frame except for its finite span. Nevertheless, the frame theoretic convolutional gridding method is still effective.
\begin{rem}
When jittered sampling is used, we do not expect to see much difference between the results using the traditional or frame theoretic convolutional gridding algorithm.  This is because the DCFs look very similar in this case, see Figure \ref{fig:dcf:jit}.  On the other hand, the log sampling case yields a substantial difference in the DCFs, and consequently in the numerical results.   As noted previously, the iterative DCF algorithm yields qualitatively similar results, but is more computationally intensive, and is only established for the traditional (linear) convolutional gridding algorithm (e.g. $r = 1$). 
\end{rem}

Figure \ref{fig1c} compares the log approximation error using the frame theoretic convolutional gridding algorithm with $r = 1, \mathcal{O}(\log n),$ and the Fourier frame approximation for  $n  =16, 32, 64, 128$.   Observe that for $r = 1$, that is, traditional convolutional gridding, the method fails to converge as $n$ increases.  In fact the method may not converge for any fixed $r$.  The method does converge (assuming no round-off error) for $r = \mathcal{O}(\log n)$, however.
\begin{figure}[htbp]
\subfloat[Jittered sampling]{\includegraphics[scale=0.35]{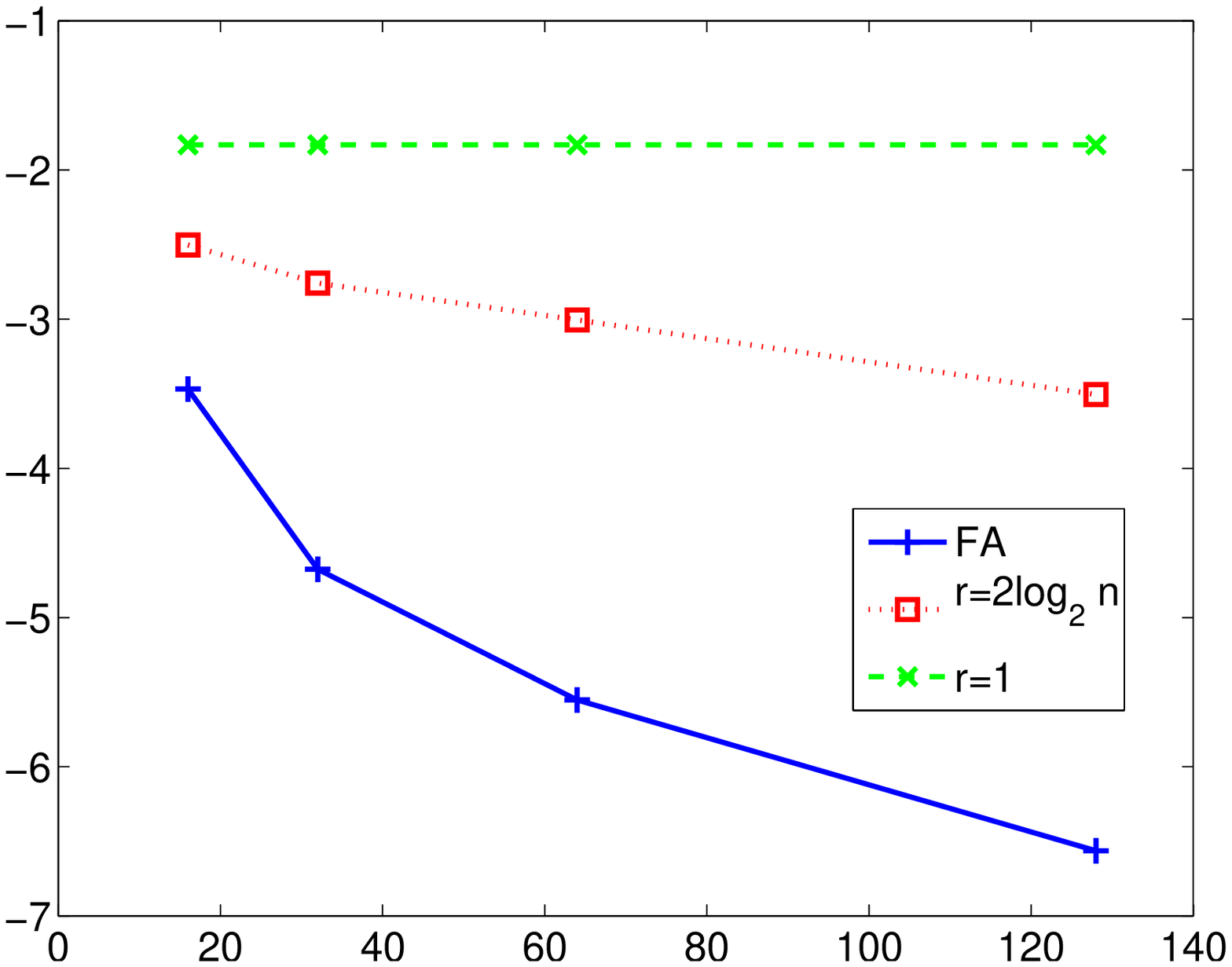}}
\qquad
\subfloat[Log sampling]{\includegraphics[scale=0.35]{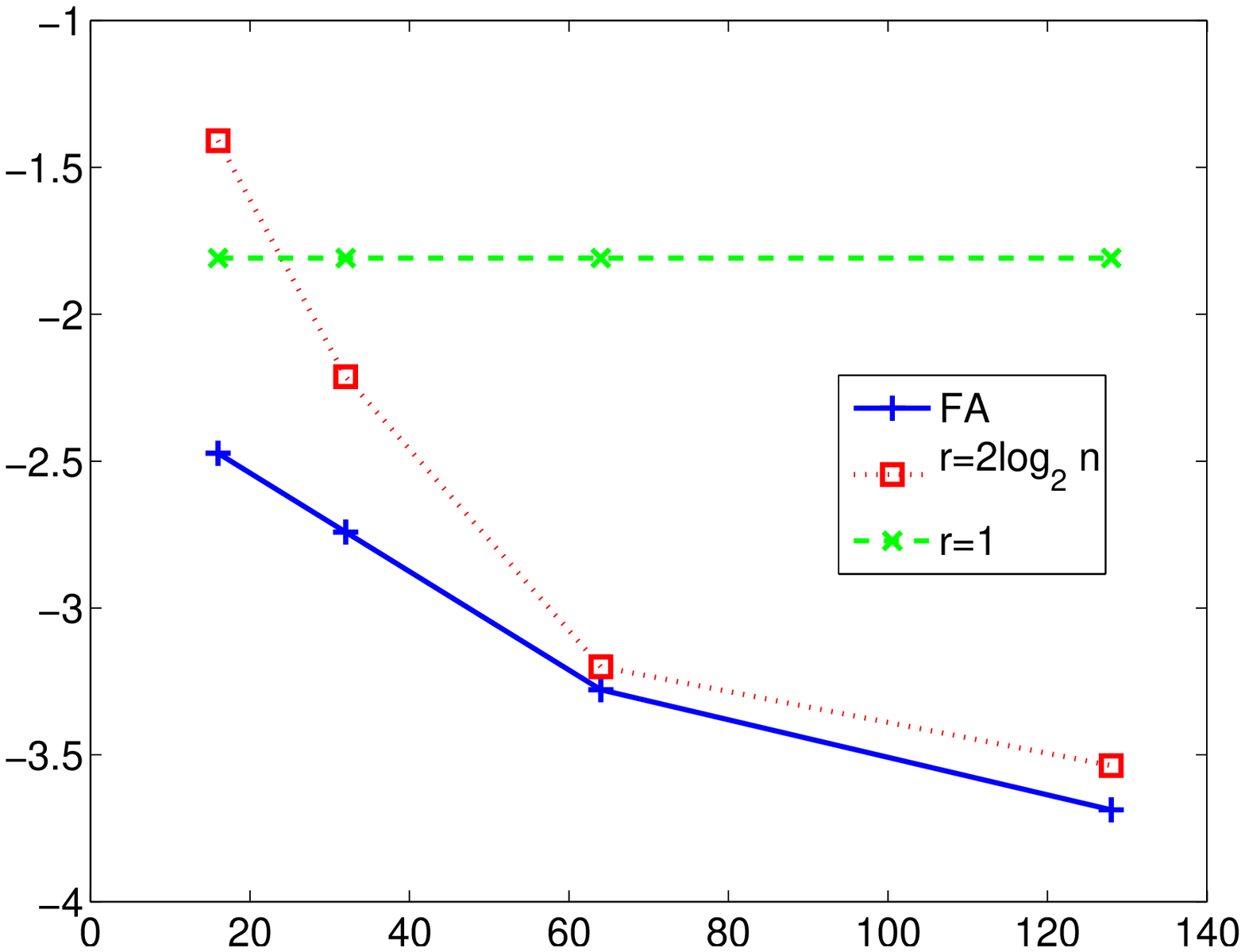}}
\caption{Approximation errors (in log) for different choices of $r$.}\label{fig1c}
\end{figure}

We now demonstrate the frame theoretic convolutional gridding method as a means of detecting edges from given non-uniform Fourier data.  Consider the test function
\begin{example}\label{ex:example2}
{\rm We consider the following piecewise continuous function}
\begin{equation*}
f(x)=\left\lbrace
\begin{array}{l l}
3/2, & 1/8\leq x\leq 1/4;\\
\frac{7}{4} - \frac{x}{2} + \sin(2\pi x -1/4),& 3/8\leq x\leq 9/16;\\
\frac{11}{4}x -5, & 11/16\leq x\leq 7/8;\\
0, & {\rm otherwise}, 
\end{array}
 \right.
\end{equation*}
{\rm for which we are given Fourier coefficients at either the jittered or log samples}.   
\end{example}
As an approximate indicator function $h$ in \eqref{eq:bump} we choose 
\begin{equation*}
h(x)=\se^{-5x^2}, \quad \epsilon_n =0.02.
\end{equation*}
Figure \ref{fig2a} shows Example \ref{ex:example2} and its corresponding target edge function, $[f](\xi)h_n(x)$ as in \eqref{eq:jumpfunction_reg}.
\begin{figure}[htbp]
\subfloat[Example \ref{ex:example2}]{\includegraphics[scale=0.35]{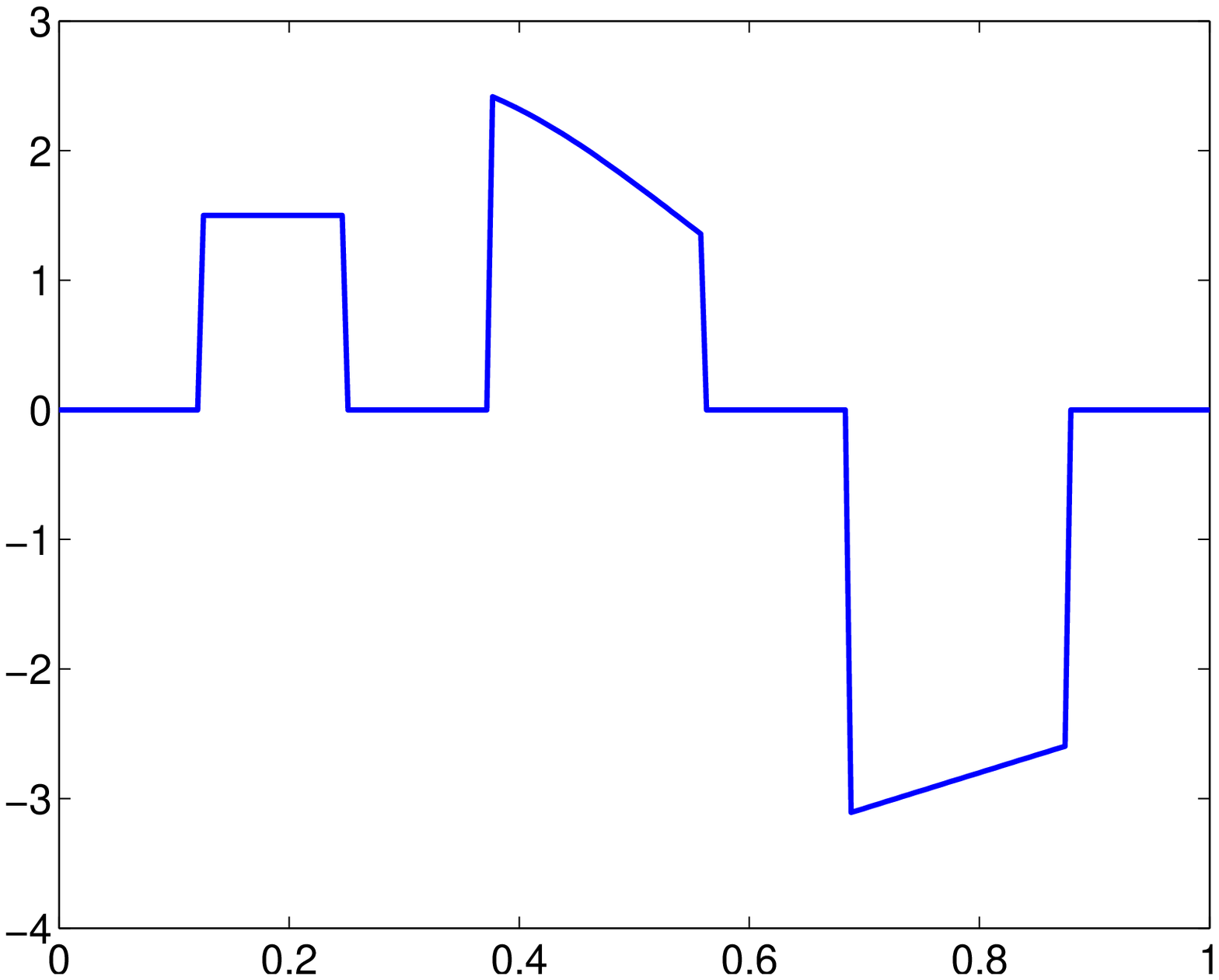}}
\qquad
\subfloat[The target edge function]{\includegraphics[scale=0.35]{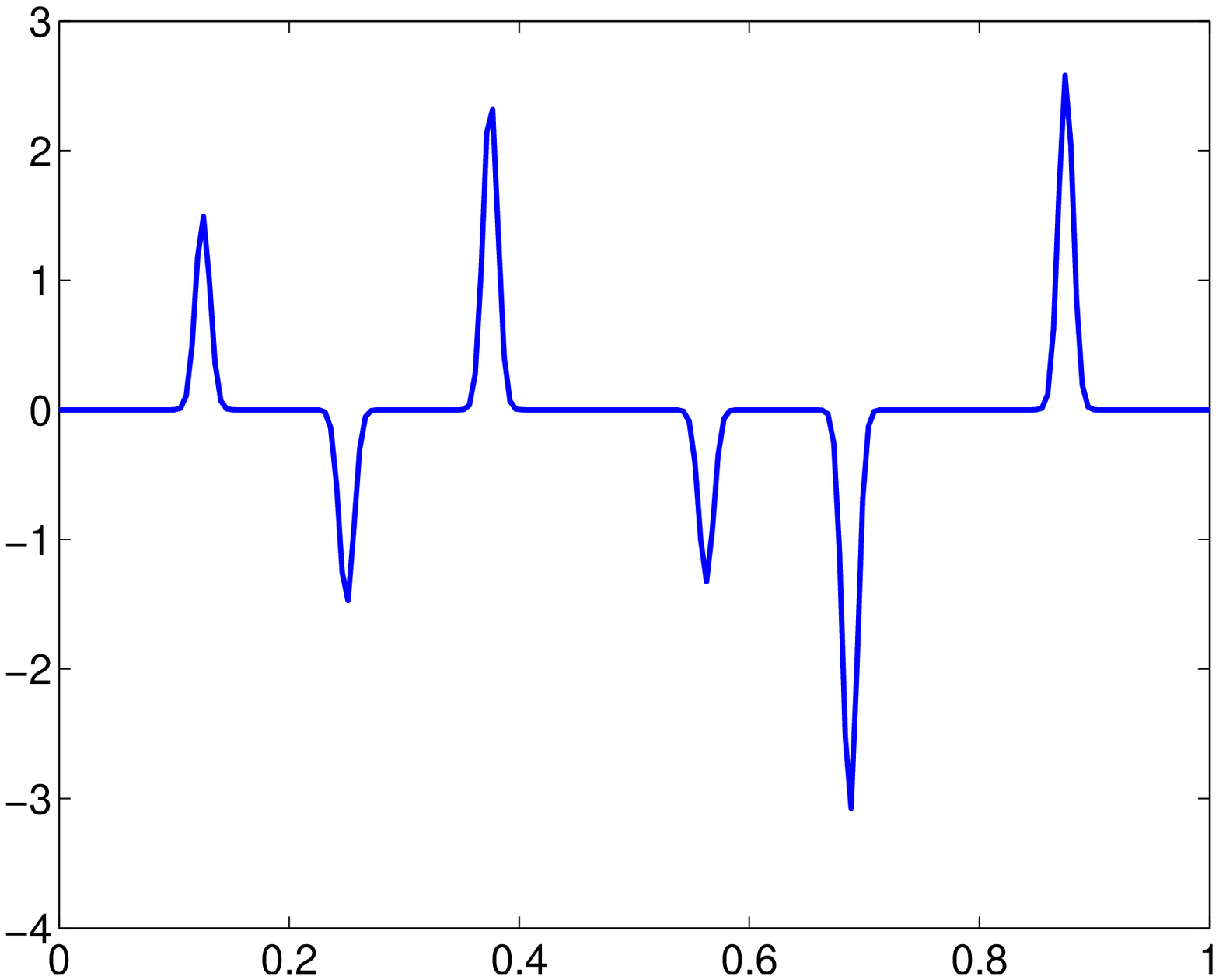}}
\label{fig2a}
\end{figure}

Figures \ref{fig2b} and \ref{fig2c}  compare the recovered edge functions and the approximation errors to the target edge function for given $n = 512$ jittered and log spaced Fourier samples respectively.  Once again we use the window function from (\ref{eq:windowfunction}) for the convolutional gridding and frame theoretic convolutional gridding, for which we chose $r = 25$.  Fewer oscillations in smooth regions are evident using the frame theoretic convolutional gridding.

\begin{figure}[htbp]
\subfloat[Edge functions]{\includegraphics[scale=0.35]{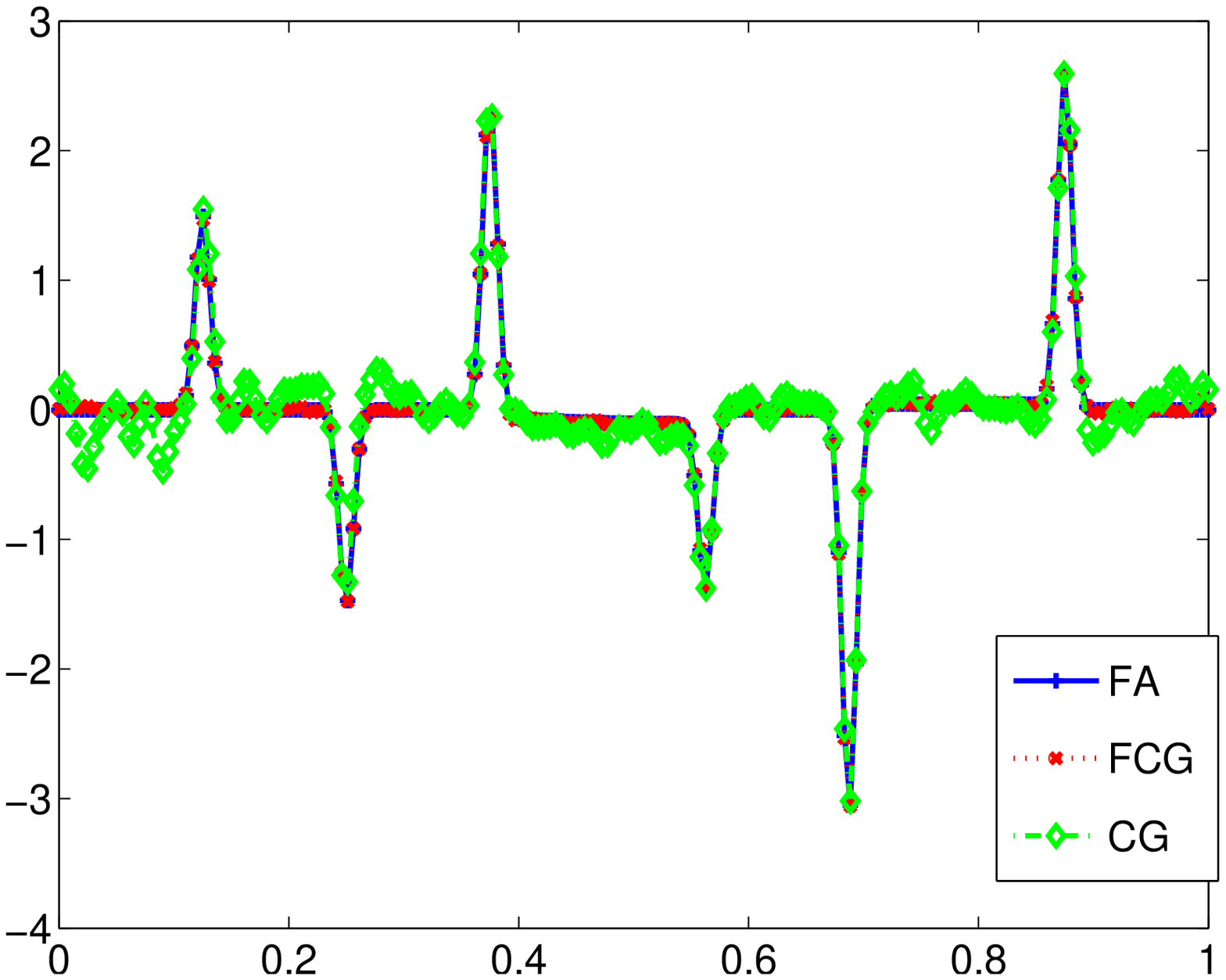}}
\qquad
\subfloat[Approximation errors]{\includegraphics[scale=0.35]{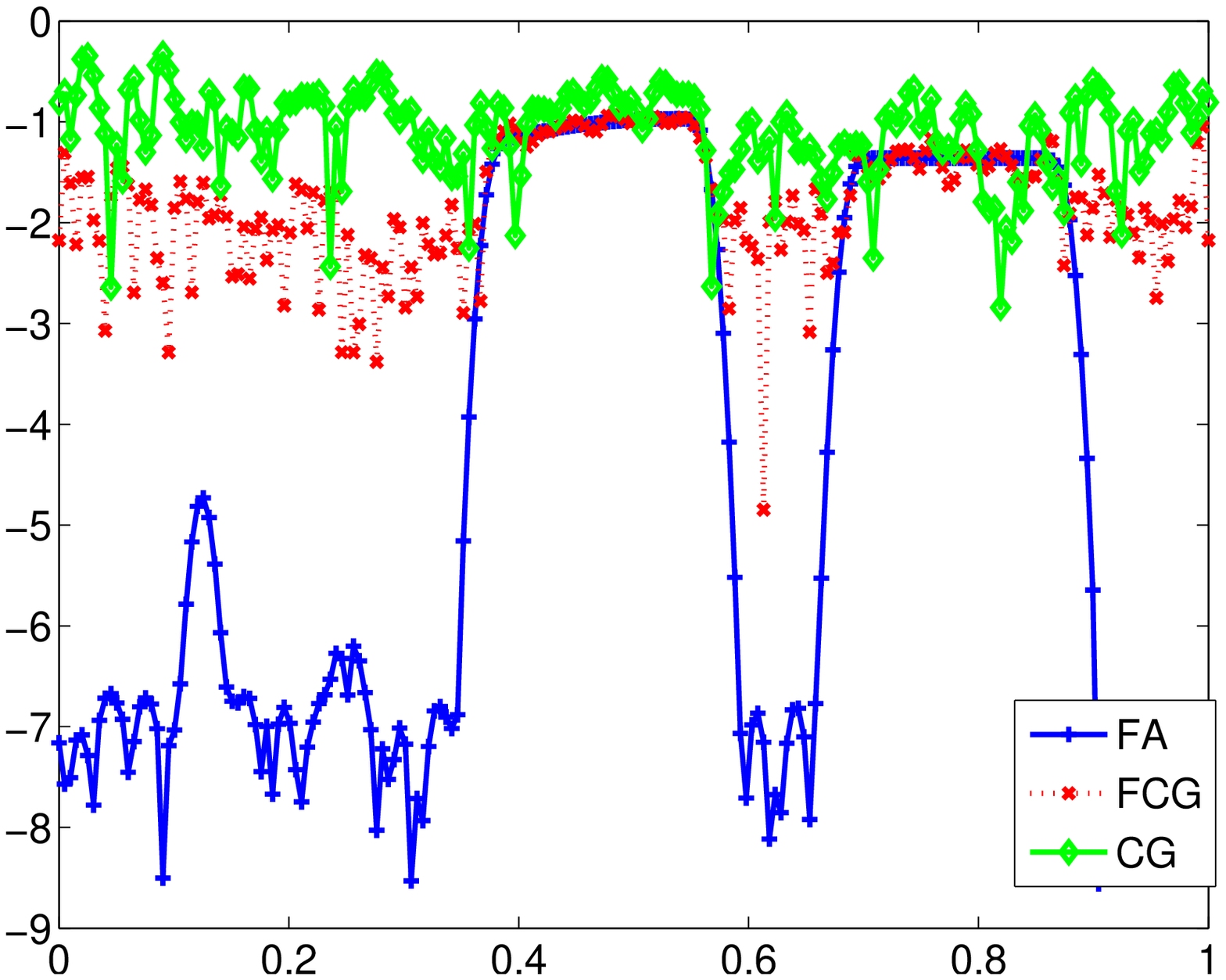}}
\caption{Comparisons of the recovered edge functions and the pointwise approximation errors (in log) to the target edge function for jittered sampling.}\label{fig2b}
\end{figure}

\begin{figure}[htbp]
\subfloat[Edge functions]{\includegraphics[scale=0.35]{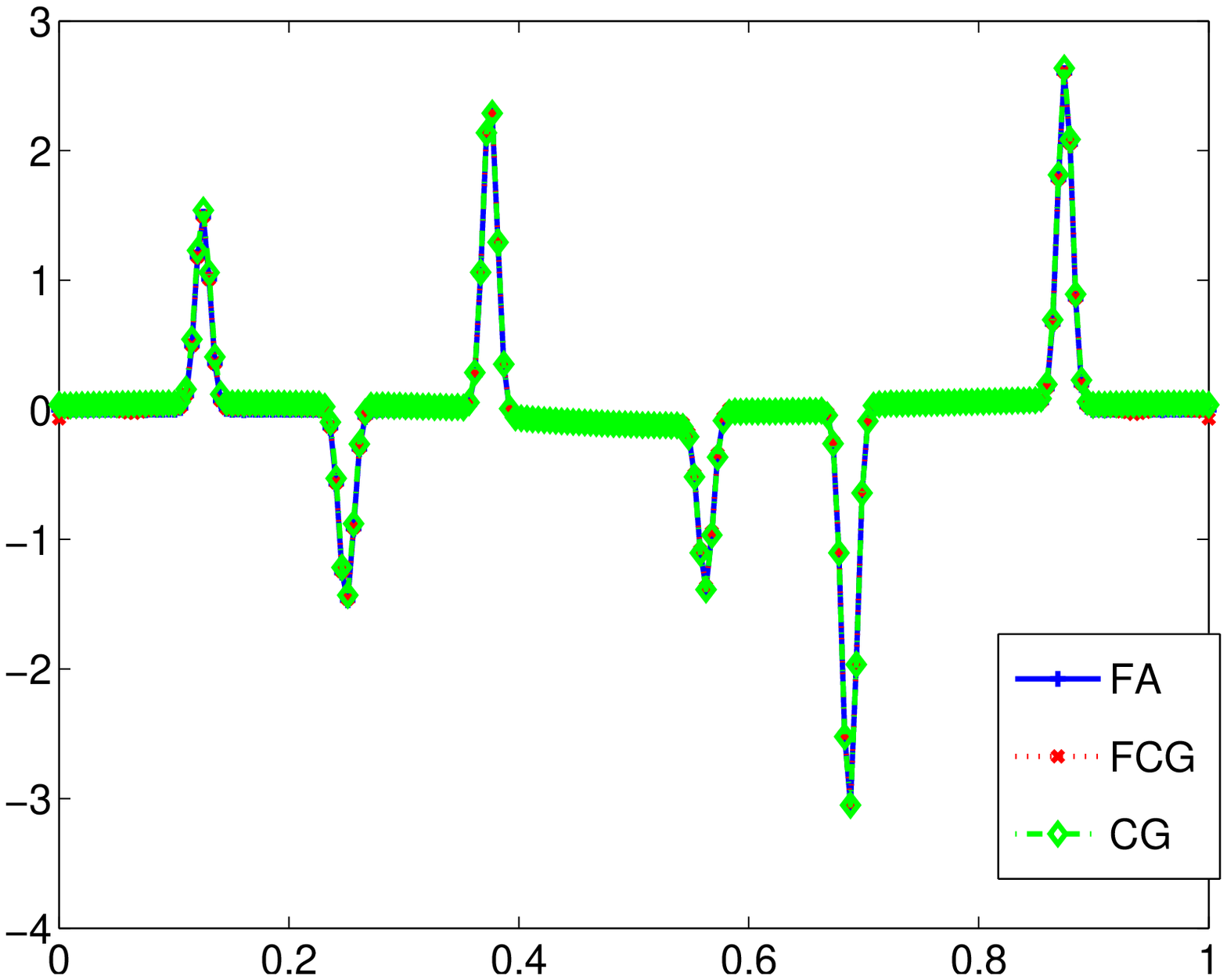}}
\qquad
\subfloat[Approximation errors]{\includegraphics[scale=0.35]{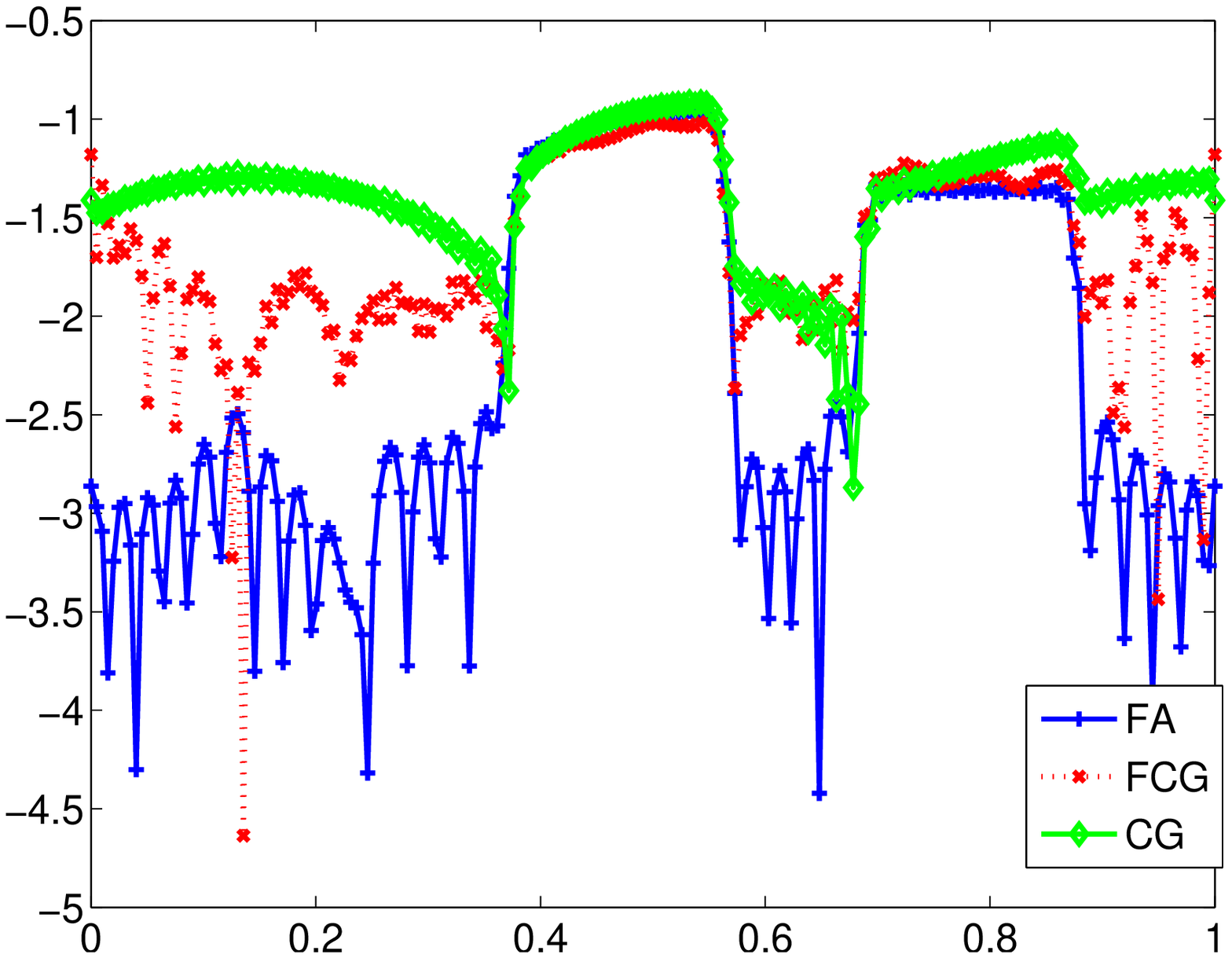}}
\caption{Comparisons of the recovered edge functions and the pointwise approximation errors (in log) to the target edge function for Log sampling.}\label{fig2c}
\end{figure}

We note that it is not straightforward to analyze the edge detection results.  In addition to the flexibility in choosing the window function and truncation parameters associated with convolutional gridding, the indicator function approximation (\ref{eq:jumpfunction_reg}) and additional error incurred in  (\ref{eq:integration}) can contribute to the overall approximation.
Further, the results displayed in Figures \ref{fig2b} and  \ref{fig2c} can be improved by post-processing.  For example, the sparsity enforcing procedure introduced in \cite{SVGR}, with some modifications, should be applicable.

\section{Conclusion}
\label{sec:conclusion}
By viewing the convolutional gridding (non-uniform FFT) method as a frame approximation, we are able to perform rigorous error analysis.  This enables us to choose optimal density compensation factors (DCFs) for the convolutional gridding method from a frame theoretic rather than heuristic perspective.  Furthermore, since the traditional (linear) convolutional gridding method fails to converge, we generalize our approach by defining the DCFs as the elements of a banded matrix instead of a vector, thereby increasing the overall accuracy of the approximation and generating numerical convergence.  As the band increases, we get closer to the frame approximation of the underlying function.   However, the computational cost for a bandwidth proportional to the log of the number of given Fourier samples remains of the same order as the traditional convolutional gridding algorithm.  Our numerical results are consistent with our theoretical results, and even in the case where the underlying sampling pattern does not generate a frame, we are able to see improvement using our method. Finally, it is important to note that the frame theoretic convolutional gridding method is a forward algorithm, and thus avoids the computational costs necessarily incurred by solving an inverse problem.

There are a large number of parameters in the convolutional gridding algorithm that can be modified for any particular circumstance.  For example, greater accuracy is possible if a larger truncation number $q$ is used in Algorithm \ref{alg:convgridding}, but this would increase the number of calculations.  There also may be better ways to choose the admissible frame, which may well lead to other types of ``regridding'' algorithms that are not FFT compatible.  In edge detection applications, using different regularizations for the indicator function, (\ref{eq:bump}), may improve the results.  However, doing so may also mean that a larger number of Fourier samples are needed to resolve the edges.
We will explore all of these ideas in future work.  
Finally, the convolutional gridding algorithm extends analytically and numerically to two dimensions.  The parameters of the method, however, are still chosen heuristically.  In future work we will investigate extending the use of admissible frames, \cite{SongGelb2013}, to two dimensions, which may enable us to choose these parameters from a frame theoretic perspective.

\bibliographystyle{siam}
\bibliography{conv_grid}

\end{document}